\documentclass[11pt]{amsart}
\usepackage{amsmath}
\usepackage{amscd}
\usepackage{amssymb}
\usepackage{pdfsync}
 \usepackage[all,cmtip]{xy}
\usepackage{graphicx}
 \usepackage{placeins}
 \usepackage{float}

\newtheorem{theorem}{Theorem}[section]

\newtheorem{lemma}[theorem]{Lemma}
\newtheorem{proposition}[theorem]{Proposition}

\theoremstyle{definition}

\newcommand{\cB}{{\mathcal B}}

\newcommand{\cF}{{\mathcal F}}
\newcommand{\cN}{{\mathcal N}}

\newcommand{\cP}{{\mathcal P}}

\newcommand{\cU}{{\mathcal U}}

\newcommand{\tr}{{\text{tr}}}

\newcommand{\spn}{{\text{span}}}
\newcommand{\im}{{\text{Im}}}
\newcommand{\cl}{{\text{cl}}}

\newcommand{\rc}{{\mathrm c}}

\newcommand{\Lb}{{\mathfrak{b}}}
\newcommand{\Lt}{{\mathfrak{t}}}
\newcommand{\Lg}{{\mathfrak g}}
\newcommand{\Lo}{{\mathfrak o}}
\newcommand{\Ln}{{\mathfrak{n}}}

\newcommand{\Ll}{{\mathfrak{l}}}
\newcommand{\Lp}{{\mathfrak{p}}}

\newcommand{\LX}{{\mathfrak{X}}}
\newcommand{\LN}{{\mathfrak{N}}}

\newcommand{\tF}{{\textbf{F}}}
\newcommand{\tk}{{\textbf{k}}}

\newcommand{\p}{\perp}

\newcommand{\la}{\langle}
\newcommand{\ra}{\rangle}

\newcommand{\beq}{\begin{equation*}}
\newcommand{\eeq}{\end{equation*}}

\begin{document}
\title[Nilpotent coadjoint orbits in small characteristic]{Nilpotent coadjoint orbits in small characteristic}
        \author{Ting Xue}
        \address{Department of Mathematics, Northwestern University,
Evanston, IL 60208, USA}
        \email{txue@math.northwestern.edu}

\maketitle
\begin{abstract}
We show that the numbers of nilpotent coadjoint orbits  in the dual of  exceptional Lie algebra $G_2$ in characteristic $3$ and in the dual of exceptional Lie algebra $F_4$ in characteristic $2$ are finite. We determine the closure relation among nilpotent coadjoint orbits in the dual of  Lie algebras of type $B,C,F_4$ in characteristic $2$ and in the dual of  Lie algebra  of type $G_2$ in characteristic $3$. In each case we
give an explicit description of the nilpotent pieces in the dual defined in \cite{CP}, which are in general unions of nilpotent coadjoint orbits, coincide with the earlier case-by-case definition in \cite{L4,X4} in the case of classical groups and have nice properties independent of the characteristic of the base field. This completes the classification of nilpotent coadjoint orbits in the dual of Lie algebras of reductive algebraic groups and the determination of closure relation among such orbits in all characteristic.
\end{abstract}


\section{introduction}
Let $G$ be a connected reductive algebraic group defined over an algebraically closed field $\tk$ of characteristic $p\geq 0$. Let $\Lg$ be the Lie algebra of $G$ and $\Lg^*$ the dual vector space of $\Lg$. Denote by $\cN_{\Lg^*}$ the set of nilpotent elements in $\Lg^*$ (recall that an element $\xi:\Lg\to\tk^*$ is called nilpotent if it annihilates some Borel subalgebra of $\Lg$, see \cite{KW}). Note that $G$ acts on ${\Lg^*}$ by coadjoint action. The $G$-orbits in $\cN_{\Lg^*}$ under this action are called nilpotent coadjoint orbits. Such orbits play an important role in representation theory. When $p$ is not special for $G$ (i.e. $p$ does not equal the ratio of the squared lengths of long and short roots in any
irreducible component of the root system of $G$), there exists a $G$-equivariant bijection between the nilpotent variety $\cN_\Lg$ of $\Lg$ and $\cN_\Lg^*$ (see \cite[section 5.6]{PS}). Hence in these cases nilpotent coadjoint orbits in $\Lg^*$ can be and are often identified with nilpotent orbits in $\Lg$ ($G$-orbits in $\cN_\Lg$ under adjoint action). The latter has been extensively studied in all characteristic. For example, the number of nilpotent orbits in $\Lg$ is well-known to be finite and
closure relation among nilpotent orbits in $\Lg$ has been determined.

It remains to study the nilpotent coadjoint orbits in $\Lg^*$  when $p$ is special, more specifically, when $G$ is of type $B$, $C$ or $F_4$ and $p=2$, and when $G$ is of type $G_2$ and $p=3$. The  nilpotent coadjoint orbits in type $B,C$ when $p=2$ have been classified in \cite{X5}. We give the classification in the remaining cases here. In particular, it follows from the classification that the number of nilpotent coadjoint orbits in $\Lg^*$ is finite for any $\Lg^*$ as in the beginning of the introduction. This result has appeared in the PhD thesis of the author and we include it here for completeness. We determine the closure relation among nilpotent coadjoint orbits in $\Lg^*$ when $p$ is special. For this we can and will assume that $G$ is adjoint in type $B$ and simply connected in type $C$.

In \cite{L4,X4}  Lusztig and the author give a case-by-case definition of a partition of $\cN_{\Lg^*}$ into nilpotent pieces for classical groups,  indexed by the set $\underline{\cU_{G_\mathbb{C}}}$ of unipotent orbits in the group $G_{\mathbb{C}}$ over complex numbers of the same type as $G$. In \cite[Section 7]{CP} Clarke and Premet define nilpotent pieces in ${\Lg^*}$ uniformly across all types,  in the same way as Lusztig's original definition of unipotent pieces (a partition of unipotent variety in $G$). Moreover the nilpotent pieces in \cite{CP} equal the Hesselink strata \cite{H1} on $\cN_{\Lg^*}$ (considered as the null-cone of $\Lg^*$ under coadjoint $G$-action), and coincide with the nilpotent pieces defined in \cite{L4,X4} for classical types.  These pieces have nice properties independent of $p$. On the other hand, there is a natural injective map from the set $\underline{\cU_{G_\mathbb{C}}}$ to the set $\underline{\cN_{\Lg^*}}$ of nilpotent coadjoint orbits in $\Lg^*$ given by Springer correspondence. Using this map and the closure relation on $\underline{\cN_{\Lg^*}}$ one can define a partition of $\cN_{\Lg^*}$ into locally closed pieces. We show that for classical groups these pieces are the same as nilpotent pieces defined by Lusztig and the author, and thus also the same as nilpotent pieces defined by Clarke and Premet. 
In particular we determine which  nilpotent coadjoint orbits lie in
the same piece (Proposition \ref{prop-psi}). The analogous result for nilpotent pieces in $\Lg$ is given in \cite{X2}. We also describe nilpotent pieces in $\Lg^*$ for type $G_2$ when $p=2$ and for type $F_4$ when $p=3$.

This paper is organized as follows. Sections \ref{sec-recol}-\ref{sec-comdef} study the cases when $G$ is of type $B$, $C$ and $p=2$. In Section \ref{sec-recol} we recall a natural partial order on the set $W^\wedge$ of irreducible characters of the Weyl group of $G$ (used by Spaltenstein \cite{Spal} to describe closure relation among nilpotent orbits in $\Lg$), the classification of nilpotent coadjoint orbits in $\Lg^*$, the combinatorial description of the Springer correspondence maps,  and the definition of nilpotent pieces in $\Lg^*$ given in \cite{L4,X4}.    Section \ref{sec-sf} and Section \ref{sec-ind} are preparation for Section \ref{sec-ord}, where we describe the Springer fibers at elements in $\cN_{\Lg^*}$ and induction for nilpotent coadjoint orbits in $\Lg^*$ (by an easy adaptation of \cite{LS1,Spal}) respectively. In Section \ref{sec-ord} we determine the closure relation on $\underline{\cN_{\Lg^*}}$ which turns out to correspond  to the natural partial order  on $W^\wedge$ recalled in Subsection \ref{sec-dualnt} via Springer correspondence map. In Section \ref{sec-comdef} we describe the nilpotent pieces in $\Lg^*$ explicitly. In Section \ref{sec-excep} we classify the nilpotent coadjoint orbits in $\Lg^*$ when $G$ is of type $G_2$ and $p=3$, and when $G$ is of type $F_4$ and $p=2$. We determine the closure relation on $\underline{\cN_{\Lg^*}}$ and describe the nilpotent pieces in $\Lg^*$ explicitly.

\vskip 10pt {\noindent\bf\large Acknowledgement}   The author wish to thank George Lusztig and Kari Vilonen for helpful discussions and for encouragement. Thanks are also due to the referee for carefully reading the manuscript and for many suggestions that helped improve the exposition of the paper.

\section{Notations and recollections}\label{sec-recol}
Although nilpotent coadjoint orbits in type $D$ are not under consideration, we include the information for type $D$ in Subsections \ref{sec-dualnt}-\ref{ssec-spcp} for future use in the inductive proof of Theorem \ref{mainthm} in
type $B$ case.
\subsection{A partial order on the set  of irreducible characters of Weyl groups  of type $B,C,D$} \label{sec-dualnt} 
 For a finite group $H$ we denote by $H^\wedge$  the set of irreducible characters of $H$ (over $\mathbb{C}$).

Let $\cP(n)$ denote the set $\{\lambda=(\lambda_1\geq\lambda_2\geq\cdots)\,|\,|\lambda|:=\sum_{i\geq 1}\lambda_i=n\}$ of all partitions of an integer $n$. For a partition $\lambda\in\cP(n)$ and each $j\geq 1$, we set $$\lambda_j^*=|\{\lambda_i\,|\,\lambda_i\geq j|\}\text{ and }m_\lambda(j)=\lambda_{j}^*-\lambda_{j+1}^*.$$
Let $\cP_2(n)$ denote the set $\{(\mu)(\nu)\,|\,|\mu|+|\nu|=n\}$ of pairs of partitions. If $W$ is a Weyl group of type $B_n$ or $C_n$ (resp. $D_n$), we can identify $W^\wedge$ with the set $\cP_2(n)$ (resp. the set $\{(\mu)(\nu)\in\cP_2(n)\,|\,\text{if }i\text{ is the smallest integer such that $\mu_i\neq \nu_i$, then $\nu_i<\mu_i$}\}$ with each pair $(\mu)(\mu)$ counted twice) (\cite{Lu3}). There is a natural partial order on the set $\cP_2(n)$ as follows. We say that
$$(\mu)(\nu)\leq(\mu')(\nu')$$ if $$\sum_{i\in[1,j]}(\mu_i+\nu_i)\leq\sum_{i\in[1,j]}(\mu_i'+\nu_i')\text{ and } \sum_{i\in[1,j-1]}(\mu_i+\nu_i)+\mu_j\leq\sum_{i\in[1,j-1]}(\mu_i'+\nu_i')+\mu_j'\text{ for all }j\geq 1;$$
and that $(\mu)(\nu)<(\mu')(\nu')$ if $(\mu)(\nu)\leq(\mu')(\nu')$ and $(\mu)(\nu)\neq(\mu')(\nu')$. This gives rise to a partial order on $W^\wedge$ (in the case of type $D_n$, the two degenerate characters corresponding to a pair $(\mu)(\mu)$ are incomparable).

\subsection{Classification of nilpotent coadjoint orbits in $\Lg^*$ (type $B,C,D$)}\label{ssec-dualorbits} Let $V$ be a finite dimensional vector space over $\tk$ equipped with a fixed non-degenerate symplectic from $\la,\ra$ (resp. a fixed non-degenerate quadratic form $Q$ with the associated bilinear form denoted by $\beta$).
We can assume that $G=Sp(V,\la,\ra)$ (resp. $G=SO(V,Q)$), the subgroup of $GL(V)$ (resp. identity component of the subgroup $O(V,Q)$ of $GL(V)$) that preserves $\la,\ra$ (resp.  $Q$). Then $\Lg=\mathfrak{sp}(V,\la,\ra)=\{x\in\mathfrak{gl}(V)\,|\,\la xv,w\ra+\la v,xw\ra=0\ \forall\ v,w\in V\}$ (resp. $\Lg=\Lo(V,Q)=\{x\in\mathfrak{gl}(V)\,|\,\beta( xv,v)=0\ \forall\ v\in V,\ x|_{\text{Rad(Q)}}=0\}$).

Let $\mathfrak{Q}(V)$ (resp. $\mathfrak{S}(V)$) denote the vector space of all quadratic forms $V\to\tk$ (resp. all symplectic forms $V\times V\to\tk$). We have a vector space isomorphism (see \cite{L4,X4})\\[5pt]
\indent (a) \qquad$\mathfrak{sp}(V,\la,\ra)^*\xrightarrow{\sim}\mathfrak{Q}(V),\ \xi\mapsto\alpha_\xi,\ \alpha_\xi(v)=\la v,Xv\ra$\\[5pt] \indent \quad(resp. $\Lo(V,Q)^*\xrightarrow{\sim}\mathfrak{S}(V),\ \xi\mapsto\beta_\xi,\ \beta_\xi(v,w)=\beta(Xv,w)-\beta(v,Xw)$),\\[5pt]
where $X$ is such that $\xi(x)=\tr(Xx)$ for all $x\in\mathfrak{sp}(V,\la,\ra)$ (resp. $\Lo(V,Q)$).

Suppose that $G=Sp(V,\la,\ra)$ with $\dim V=2n$ and $\xi\in\cN_{\Lg^*}$. Let $\alpha_\xi$ always denote the quadratic form corresponding to $\xi$ under the isomorphism in (a) and let $\beta_\xi$ always denote the bilinear form associated to $\alpha_\xi$ (given by $\beta_\xi(v,w)=\alpha_\xi(v+w)-\alpha_\xi(v)-\alpha_\xi(w$)). Let $T_\xi:V\to V$ be defined by $\la T_\xi v,w\ra=\beta_\xi(v,w)$ for all $v,w\in V$. Assume that $p=2$. Then the $G$-orbit of $\xi$ is characterized by a pair $(\lambda,\chi)$ as follows (\cite{X5}):\\[5pt]
\indent (i)  the partition $\lambda\in\cP(2n)$ given by the sizes of Jordan blocks of $T_\xi$ (we have $m_\lambda(i)$ even for all $i>0$);\\[5pt]
\indent (ii) the map $\chi:\mathbb{N}\to\mathbb{N}$ given by $\chi(k):=\chi_{\alpha_\xi}(k)=\min\{l\,|\,T_\xi^kv=0\Rightarrow\alpha_\xi(T_\xi^lv)=0\ \forall\ v\in V\}$ (we have  $\frac{\lambda_i-1}{2}\leq\chi(\lambda_i)\leq\lambda_i$, $\chi(\lambda_i)\geq\chi(\lambda_{i+1})$ and $\lambda_i-\chi(\lambda_i)\geq\lambda_{i+1}-\chi(\lambda_{i+1})$ for all $i\geq 1$).

Suppose that $G=SO(V,Q)$ and $\xi\in\cN_{\Lg}^*$. Let $\beta_\xi$ always denote the symplectic form  corresponding to $\xi$ under the isomorphism in (a).

If $p=2$, $G=SO(V,Q)$ and $\dim V=2n$, let $T_\xi:V\to V$ be defined by $\beta( T_\xi v,w)=\beta_\xi(v,w)$ for all $v,w\in V$. The $O(V,Q)$-orbit of $\xi$ is characterized by a pair $(\lambda,\chi)$ as follows (\cite{H2}):\\[5pt]
\indent (i)  the partition $\lambda\in\cP(2n)$ given by the sizes of Jordan blocks of $T_\xi$ (we have $m_\lambda(i)$ even for all $i>0$);\\[5pt]
\indent (ii) the map $\chi:\mathbb{N}\to\mathbb{N}$ given by $\chi(k):=\chi_{T_\xi}(k)=\min\{l\,|\,T_\xi^kv=0\Rightarrow Q(T_\xi^lv)=0\ \forall\ v\in V\}$ (we have $\frac{\lambda_i}{2}\leq\chi(\lambda_i)\leq\lambda_i$, $\chi(\lambda_i)\geq\chi(\lambda_{i+1})$ and $\lambda_i-\chi(\lambda_i)\geq\lambda_{i+1}-\chi(\lambda_{i+1})$ for all $i\geq 1$).

Assume that $p=2$, $G=SO(V,Q)$ and $\dim V=2n+1$. Let $m\in[0,n]$ be the unique integer such that there exists a (unique) set of vectors $\{v_i,i\in[0,m]\}$ with \\[5pt]
\indent (b) $Q(v_m)=1,\ Q(v_i)=0,\ \beta_\xi(v_i,v)=\beta(v_{i-1},v),\ i\in[1,m],\ \beta(v_m,v)=0,\  \beta_\xi(v_0,v)=0,\ \forall\ v\in V.$\\[5pt]
If $m=0$, let $W$ be a complementary subspace of $\spn\{v_0\}$ in $V$; if $m\geq 1$, let $\{u_i,i\in[0,m-1]\}$ be a set of vectors such that\\[5pt]
\indent (c) $Q(u_0)=0,\beta(u_0,v_j)=\delta_{j,0},j\in[0,m];\beta_\xi(u_{i-1},v)=\beta(u_{i},v),Q(u_i)=0,i\in[1,m-1],\forall\  v\in V$\\[5pt]  and let $W=\{v\in V\,|\,\beta(v,u_i)=\beta(v,v_i)=\beta_\xi(u_0,v)=0\}$. Then $V=\spn\{u_i,v_i\}\oplus W$ and $\beta|_{W}$ is non-degenerate. Define $T_\xi: W\to W$ by
$\beta_\xi(w,w')=\beta( T_\xi w,w')$ for all $w,w'\in W$ and let $\chi_W:\mathbb{N}\to\mathbb{N}$ be given by $\chi_W(k)=\min\{l\,|\,T_\xi^kw=0\Rightarrow Q(T_\xi^lw)=0\ \forall\ w\in W\}$.
Then the orbit of $\xi$ is characterized by $(m;(\lambda,\chi))$ as follows (\cite{X5}):\\[5pt]
\indent (i) the integer $m\in[0,n]$;\\[5pt]
\indent (ii) the partition $\lambda\in\cP(2n-2m)$ given by the sizes of Jordan blocks of $T_\xi$ (we have $m_\lambda(i)$ even for all $i>0$);\\[5pt]
\indent (iii)  the map $\chi:\mathbb{N}\to\mathbb{N}$ given by $\chi(i)=\max(i-m,\chi_W(i))$ (we have $m\geq\lambda_i-\chi(\lambda_i)\geq\lambda_{i+1}-\chi(\lambda_{i+1})$, $\frac{\lambda_i}{2}\leq\chi(\lambda_i)\leq\lambda_i$ and $\chi(\lambda_i)\geq\chi(\lambda_{i+1})$ for all $i\geq 1$).\\[5pt]
Note that $(m;(\lambda,\chi))$ does not depend on the choice of $W$ and $u_0$.

Let $\mathfrak{N}_{B_n}^{*2}$ (resp. $\mathfrak{N}_{C_n}^{*2}$, $\mathfrak{N}_{D_n}^{*2}$) be the set of all $(m;(\lambda,\chi))$ (resp. $(\lambda,\chi)$) corresponding to nilpotent coadjoint orbits ($p=2$) in $\mathfrak{o}(2n+1)^*$ (resp. $\mathfrak{sp}(2n)^*$, $\Lo(2n)^*$). Note that in the case of $\Lo(2n)^*$, there are two orbits corresponding to each pair $(\lambda,\chi)$ with $\chi(\lambda_i)=\frac{\lambda_i}{2}$ for all $i$. Let $$\mathfrak{N}_{B}^{*2}=\cup_{n\geq 0}\mathfrak{N}_{B_n}^{*2}, \ \ \mathfrak{N}_{C}^{*2}=\cup_{n\geq 0}\mathfrak{N}_{C_n}^{*2},\ \ \mathfrak{N}_{D}^{*2}=\cup_{n\geq 0}\mathfrak{N}_{D_n}^{*2}.$$

\subsection{Combinatorial description of Springer correspondence maps (type $B,C,D$)}\label{ssec-spcp} Let $W$ denote the Weyl group of $G$. Recall that we have an injective Springer correspondence map \cite{X5}:
$$\gamma_{\Lg^*}:\underline{\cN_{\Lg^*}}\to W^\wedge,$$
which maps an orbit $\rc$ to the Weyl group character corresponding to the pair $(\rc,1)$ under Springer correspondence.
When $p=2$, the Springer correspondence maps $\gamma_{\Lg^*}$ are given as follows (\cite{X3}) \begin{eqnarray*}&&\gamma^*_{B_n}:=\gamma_{\mathfrak{o}(2n+1)^*}:\underline{\cN_{\mathfrak{o}(2n+1)^*}}=\LN_{B_n}^{*2}\to W^\wedge,\\&& \qquad (m;(\lambda,\chi))\mapsto(\mu)(\nu),\ \mu_1=m,\ \mu_{i+1}=\lambda_{2i-1}-\chi(\lambda_{2i-1}),\ \nu_i=\chi(\lambda_{2i-1}),\ i\geq 1;\\&&\gamma^*_{C_n}:=\gamma_{\mathfrak{sp}(2n)^*}:\underline{\cN_{\mathfrak{sp}(2n)^*}}=\LN_{C_n}^{*2}\to W^\wedge,\\&&\qquad (\lambda,\chi)\mapsto(\mu)(\nu),\ \mu_i=\chi(\lambda_{2i-1}),\ \nu_i=\lambda_{2i-1}-\chi(\lambda_{2i-1}),\ i\geq 1;\\
&&\gamma^*_{D_n}:=\gamma_{\mathfrak{o}(2n)^*}:\underline{\cN_{\mathfrak{o}(2n)^*}}=\LN_{D_n}^{*2}\to W^\wedge,\\&&\qquad (\lambda,\chi)\mapsto(\mu)(\nu),\ \mu_i=\chi(\lambda_{2i-1}),\ \nu_i=\lambda_{2i-1}-\chi(\lambda_{2i-1}),\ i\geq 1.
\end{eqnarray*}
We denote the image of $\gamma_{B_n}^*$ (resp. $\gamma_{C_n}^*$, $\gamma_{D_n}^*$) (when $p=2$) by $\LX_{B_n}^{*2}$ (resp. $\LX_{C_n}^{*2}$ and $\LX_{D_n}^{*2}$). Let $\LX_{B}^{*2}=\cup_{n\geq 0}\LX_{B_n}^{*2}$, $\LX_{C}^{*2}=\cup_{n\geq 0}\LX_{C_n}^{*2}$ and $\LX_{D}^{*2}=\cup_{n\geq 0}\LX_{D_n}^{*2}$.
We have $$\LX_B^{*2}=\{(\mu)(\nu)|\nu_i\geq\mu_{i+1}\},\ \LX_{C}^{*2}=\{(\mu)(\nu)|\nu_i\leq\mu_i+1\}$$ (here we use the identification of $W^\wedge$ with $\cP_2(n)$).

\subsection{Nilpotent pieces in $\mathfrak{sp}(2n)^*$ and $\Lo(2n+1)^*$}\label{ssec-p}Suppose that $G=Sp(V,\la,\ra)$ (resp. $SO(V,Q)$) as in Subsection \ref{ssec-dualorbits}. Let $\rc\in\underline{\cN_{\Lg^*}}$ and
$\xi \in \rc$. Let $V_*=(V_{\geq a})_{a\in\mathbb{Z}}$ be the canonical filtration of $V$ associated to $\xi$ as in \cite{L4,X4}, where $V_{\geq a+1}\subset V_{\geq
a}\subset V$. If $p\neq 2$, let $T_\xi$ be defined as in Subsection \ref{ssec-dualorbits} (resp. by $\beta( T_\xi v, w)=\beta_\xi(v,w)$ for all $v,w\in V$), then\\[5pt]
\indent\qquad (a) {\em $V_{\geq a}=\sum_{j\geq\max(0,a)}T_\xi^j(\ker T_\xi^{2j-a+1})$.}\\[5pt]
The definitions of $V_*$ when $p=2$ are recalled in  \ref{lem-sp1} (resp. \ref{lem-o1}). We
define $$f_a=\dim V_{\geq a}/V_{\geq a+1}.$$
Then
 $f_a\neq 0$ for finitely many $a\in\mathbb{Z}$ and $f_{-a}=f_a$.
 The sequence of numbers $(f_a)_{a\in\mathbb{N}}$ ($\mathbb{N}=\{0,1,2,\ldots\}$)  depends only on $\rc$ and not
on the choice of $\xi\in\rc$; we denote this sequence  by
$\Upsilon_\rc$. Two sequences $(f_a)_{a\in\mathbb{N}}$ and $(h_a)_{a\in\mathbb{N}}$ are equal iff $f_a=h_a$ for all $a\in\mathbb{N}$.

\begin{lemma}[\cite{L4,X4}]\label{lempiece}
The orbits $\rc_1,\rc_2\in\underline{\cN_{\Lg^*}}$ lie in the same piece if and only if
$\Upsilon_{\rc_1}=\Upsilon_{\rc_2}$.
\end{lemma}
Note that if $p\neq 2$, the orbit of  $\xi\in\cN_{\Lg^*}$ is characterized  by the partition $\lambda$ given by the sizes of Jordan blocks of $T_\xi$. It follows from  (a) that\\[5pt]
\indent\qquad (a$'$) {\em if $p\neq2$, $\rc=\lambda$ and $\Upsilon_\rc=(f_a)_{a\in\mathbb{N}}$, then $f_a=\sum_{i\in\mathbb{N}}m_{\lambda}(a+2i+1)\text{ for all }a\in\mathbb{N}$.}\\[5pt]
Moreover each nilpotent piece consists of one orbit when $p\neq 2$.

\subsubsection{Canonical filtrations $V_*$  for $\xi\in\cN_{\mathfrak{sp}(2n)^*}$ ($p=2$)}\label{lem-sp1}
Assume that $G=Sp(V,\la,\ra)$ and $p=2$. Let
$\xi\in \cN_{\Lg^*}$ and let $\alpha_\xi,\beta_\xi,T_\xi$ be defined for $\xi$ as in Subsection \ref{ssec-dualorbits}. The canonical filtration $V_*=(V_{\geq a})$ associated to $\xi$ is defined by induction on $\dim V$ as follows  (\cite{L4}),  where $V_{\geq a}=V_{\geq 1-a}^\p$. If $\xi=0$, we set $V_{\geq a}=0$ for all $a\geq 1$ and $V_{\geq a}=V$ for all $a\leq0$. Hence $V_*$ is defined when $\dim V\leq 1$. Assume now that $\xi\neq 0$ and $\dim V\geq 2$. Let $e$ be the smallest integer such that $T_\xi^eV=0$, $f$ the smallest integer such that $\alpha_\xi (T_\xi^fV)=0$ and\\[5pt]
\indent\hspace{1.5in}$N=\max(e-1,2f-1).$\\[5pt] Then $N\geq 1$. We set
\begin{eqnarray*}&&V_{\geq a}=V\text{ for all } a\leq-N;\ V_{\geq a}=0\text{ for all }a\geq N+1;\\&&V_{\geq -N+1}=\left\{\begin{array}{ll}\{v\in V|T_\xi^{e-1}v=0\}&\text{ if }e=2f+1\\\{v\in V|T_\xi^{e-1}v=0, \alpha_\xi(T_\xi^{f-1}v)=0\}&\text{ if }e=2f\\
\{v\in V|\alpha_\xi(T_\xi^{f-1}v)=0\}&\text{ if }e<2f\end{array}\right.;\ V_{\geq N}=V_{\geq -N+1}^\perp.
\end{eqnarray*}
Let $V'=V_{\geq -N+1}/V_{\geq N}$. Then $\la,\ra$ induces a nondegenerate symplectic form $\la,\ra'$ on $V'$ and $\alpha_\xi$ induces a quadratic form $\alpha_{\xi'}$ corresponding to $\xi'\in \cN_{\Lg'^*}$, where $\Lg'=\mathfrak{sp}(V',\la,\ra')$. By induction hypothesis, a canonical filtration $V'_*=(V'_{\geq a})$ of $V'$ is defined for $\xi'$. For $a\in[-N+1,N]$ we set $V_{\geq a}$ to be the inverse image of $V'_{\geq a}$ under the natural map $V_{\geq -N+1}\to V'$ (note that $V'_{\geq N}=0$ and $V'_{\geq -N+1}=V'$). This completes the definition of $V_*$.

\subsubsection{}\label{lem-sp2}
Suppose that $p=2$ and the $G$-orbit of $\xi\in\cN_{\mathfrak{sp}(2n)^*}$ corresponds to $(\lambda,\chi)\in\LN^{*2}_C$. Recall that (\cite{X5}) we have a decomposition $V=\oplus_{a\in[1,r]}W(a)$ of $V$ into mutually orthogonal $T_\xi$-stable subspaces  such that $m_\lambda(i)=\sum _{a\in[1,r]}m_{\lambda^a}(i)$, $\chi(i)=\max_a\chi_a(i)$, where $\alpha_\xi|_{W(a)}=(\lambda^a,\chi_a)$. Moreover, $\alpha_\xi|_{W(a)}={}^*W_{\chi(\lambda_{2a})}(\lambda_{2a})$, $a\in[1,r]$,
where\\[5pt]
\indent (a) $\alpha_\xi|_W={}^*W_{l}(s)$ means that  there exist $v,w\in W$ such that $W=\text{span}\{T_\xi^iv,T_\xi^{i}w,i\in[0,s-1]\}, \ \la T_\xi^iv,w\ra=\delta_{i,s-1},\  \alpha_\xi(T_\xi^iv)=\delta_{i,l-1},\ \alpha_\xi(T_\xi^iw)=0$; we have $\chi_{\alpha_\xi|_W}(i)=\max(0,\min(i-s+l,l))$.

We state some facts  which will be used  later (see \cite{X5}).\\[5pt]
\indent(i) Let $W$ be a $T_\xi$-stable subspace of $V$ such that $\alpha_\xi|_W={}^*W_f(e-j)$ with  $f>\frac{e-j}{2}$. Let  $K_W=\{v\in W|\alpha_\xi(T_\xi^{f-1}v)=0\}$, $L_W=K_W^{\p}\cap W$, $W'=K_W/L_W$ and let $\alpha_{\xi'}$ be the quadratic form on $W'$ induced by $\alpha_\xi$. Using the basis for $W$ chosen as in (a), one can easily check that\\[5pt]  \indent\hspace{1.5in}$\alpha_{\xi'}|_{W'}={}^*W_{f-1}(e-j-1).$\\[5pt]
\indent(ii) Let $W$ be a $T_\xi$-stable subspace of $V$ such that $\alpha_\xi|_W={}^*W_f(e)^{a}$ (an orthogonal decomposition into $a$ copies of ${}^*W_f(e)$), where $f\leq\frac{e}{2}$. Let $W'=(\ker T_\xi^{e-1}\cap W)/T_\xi^{e-1}W$ and $\alpha_{\xi'}$ be the quadratic form  on $W'$ induced by $\alpha_\xi$. Using the basis for $W$ chosen as in (a), one can easily check that  \\[5pt]\indent\hspace{1.5in}$\alpha_{\xi'}|_{W'}={}^*W_{f-1}(e-2)^{a}.$

\subsubsection{Canonical filtrations $V_*$  for $\xi\in\cN_{\mathfrak{o}({2n+1})^*}$ ($p=2$)}\label{lem-o1}
Assume that $G=SO(V,Q)$ and $p=2$. Let
$\xi\in \cN_{\Lg^*}$ and let $\beta_\xi$, $m,\{v_i\},\{u_i\},W,T_\xi,\chi$ be defined for $\xi$ as in Subsection \ref{ssec-dualorbits}. The canonical filtration $V_*=(V_{\geq a})$  associated to $\xi$ is defined  by induction on $\dim V$ as follows  (see
\cite{X4}), where $V_{\geq 1-a}=V_{\geq a}^\p$ and $Q|_{V_{\geq a}}=0$ for all $a\geq 1$. If $\xi=0$ we set $V_{\geq a}=0$ for all $a\geq 1$ and $V_{\geq a}=V$ for all $a\leq0$. Hence $V_*$ is defined when $\dim V\leq 1$. Assume now that $\xi\neq 0$ and $\dim V\geq 2$. Let $e$ be the smallest integer such that $T_\xi^eW=0$, $f=\chi(e)$ and\\[5pt]\indent\hspace{1.5in}$N=\max(2m,m+f-1).$\\[5pt] Then $N\geq 1$. We set
\begin{eqnarray*}&&V_{\geq a}=V\text{ for all } a\leq-N;\ V_{\geq a}=0\text{ for all }a\geq N+1;\ V_{\geq N}=V_{\geq -N+1}^\perp\cap Q^{-1}(0),\\
&&V_{\geq -N+1}=\left\{\begin{array}{ll}\spn\{v_m\}\oplus\ker T_\xi^{e-1}&\text{if }m=0,\\[4pt]\spn\{v_i,i\in[0,m],u_i,i\in[1,m-1]\}\oplus W& \text{if } m\geq f,\\[4pt]
\spn\{v_i,i\in[0,m],u_i,i\in[1,m-1]\}\\\oplus\{w\in W|Q(T_\xi^{f-1}w)=0\}& \text{if }e-f<m<f,\\[4pt]
\spn\{v_i,i\in[0,m],u_i,i\in[1,m-1]\}\\\oplus\{w\in W|T_\xi^{e-1}w=0,Q(T_\xi^{f-1}w)=0\} &\text{if }0<m=e-f=f-1, \\[4pt]
&\text{or }0<m=e-f<f-1\text{ and }\rho\neq 0,\\[4pt]
\spn\{v_i,i\in[0,m],u_i,i\in[1,m-1]\}\\\oplus \spn\{u_0+w_{**}\}\oplus\ker T_\xi^{e-1}
& \text{if }0<m=e-f<f-1\text{ and }\rho=0,\\\end{array}\right.
\end{eqnarray*}
where $\rho:\ker T_\xi^{e-1}\to\tk$ is the map $w\mapsto Q(T_\xi^{f-1}w)$ and $w_{**}\in W$ is  such that $\beta({T_\xi^{e-1}w_{**},w})^2=Q(T_\xi^{f-1}w)$ for all $w\in W$.

Let $V'=V_{\geq -N+1}/V_{\geq N}$. Then $Q$ induces a non-degenerate quadratic form $Q'$ on $V'$ and $\beta_\xi$ induces a symplectic form $\beta_{\xi'}$ corresponding to  $\xi'\in\cN_{\Lg'^*}$, where $\Lg'=\mathfrak{o}(V',Q')$. By induction hypothesis, a canonical filtration $V'_*=(V'_{\geq a})$ of $V'$ is defined for $\xi'$. For $a\in[-N+1,N]$ we set $V_{\geq a}$ to be the inverse image of $V'_{\geq a}$ under the natural map $V_{\geq -N+1}\to V'$ (note that $V'_{\geq N}=0$ and $V'_{\geq -N+1}=V'$). This completes the definition of $V_*$.

\section{Springer fibers}\label{sec-sf}

Suppose that $p=2$ and $G$ is of type $B$ or $C$ in this section.

\subsection{}
For a Borel subgroup $B$ of $G$, we
denote by $\Lb$ the Lie algebra of $B$ and define $\Ln^*=\{\xi\in\Lg^*\,|\,\xi(\Lb)=0\}$.
For $\xi\in\cN_{\Lg^*}$, denote by $cl_G(\xi)$ the $G$-orbit of $\xi$ and by $Z_G(\xi)$ the centralizer of $\xi$ in $G$. Let $\cB^G$ be the variety of all Borel subgroups of $G$ and
 let
$\cB^G_\xi=\{B\in\cB^G\,|\,\xi\in\Ln^*\}$ be the Springer fiber at $\xi\in\cN_{\Lg^*}$.  One can easily adapt the proofs in \cite[p. 132]{st} and \cite{Spa2} to show that $\cB_
\xi^G$ is connected and all its irreducible components have the same dimension.
\begin{proposition}\label{prop-sf}
Suppose that $\xi\in\cN_{\Lg^*}$. We have $\displaystyle{\dim\cB^G_\xi=\frac{{\dim Z_G(\xi)-\text{rank}\ G}}{2}}.$
\end{proposition}
The proofs of the proposition for $G=Sp(2n)$ and $G=SO(2n+1)$  are given in Subsection \ref{op-sp} and  Subsection \ref{sec-op-o} respectively.
\subsection{}\label{op-sp}Suppose that $G=Sp(V,\la,\ra)$ with $\dim V=2n$ in this subsection. We can identify the variety $\cB^G$ with the set $\cF^G$ of all complete flags $0=V_0\subset V_1\subset \cdots\subset V_{2n}=V$ such that $V_{2n-i}=V_i^\p$ for all $ i\leq n$.
\begin{lemma} Suppose that $\xi\in\cN_{\Lg^*}$. We can identify $\cB_\xi^G$ with the set
$$\cF_\xi^G=\{F=(V_0\subset V_1\subset\cdots\subset V_{2n})\in\cF^G\,|\,{\beta_\xi(V_i,V_{2n+1-i})=0}\text{ and }\alpha_\xi(V_i)=0\text{ for all }i\leq n\}.$$
\end{lemma}
\begin{proof}
Let $\xi\in\mathfrak{sp}(2n)^*$, $x\in\mathfrak{sp}(2n)$ and $F=(V_0\subset V_1\subset \cdots\subset V_{2n})\in\cF^G$.  There exists a basis $e_i$, $i\in[-n,n]-\{0\}$, of $V$ such that $\la e_i,e_j\ra=\delta_{i+j,0}$ and $V_k=\spn\{e_i,\ i\in[1,k]\}$ for all $k\leq n$. Assume that $xe_i=\sum_{j}x_{ij}e_j$. We have $x_{i,-j}+x_{j,-i}=0$ and
\begin{eqnarray*}&&\xi(x)=\sum_{i,j\in[1,n]}x_{ji}\beta_\xi(e_i,e_{-j})+\sum_{1\leq i<j\leq n}x_{j,-i}\beta_\xi(e_{-j},e_{-i})+\sum_{1\leq i<j\leq n}x_{-i,j}\beta_\xi(e_j,e_i)\\&&\qquad\qquad+\sum_{i\in[1,n]}x_{i,-i}\alpha_\xi(e_{-i})+\sum_{i\in[1,n]}x_{-i,i}\alpha_\xi(e_i).\end{eqnarray*}
Let $B=\{g\in G|g V_i=V_i\}$ be the Borel subgroup corresponding to $F$. For $x\in\Lb$, we have $x_{ij}=0$ for $1\leq i<j\leq n$, and $x_{i,-j}=0$ for $i,j\in[1,n]$. Hence $\xi\in\Ln^*$ if and only if $\beta_\xi(e_i,e_{-j})=0$ for all $1\leq i\leq j\leq n$, $\beta_\xi(e_j,e_i)=0$ for all $1\leq i<j\leq n$, and $\alpha_\xi(e_i)=0$ for all $i\in[1,n]$. The lemma follows.
\end{proof}
 Let $\xi\in\cN_{\Lg^*}$ and suppose that  $cl_G(\xi)=(\lambda,\chi)\in\LN_{C_{n}}^{*2}$. Let $F=(V_i)\in\cF_\xi^G$. Then we have $\alpha_\xi|_{V_1}=0$ and $\beta_\xi(V_1,V)=0$. Let $V'=V_1^\p/V_1$. The non-degenerate symplectic form $\la,\ra$ on $V$  induces a non-degenerate symplectic form $\la,\ra'$ on $V'$; we write $\Lg'=\mathfrak{sp}(V',\la,\ra')$. Moreover the quadratic form $\alpha_\xi$ induces a quadratic form $\alpha_{\xi'}$ on $V'$ which corresponds to an orbit $(\lambda',\chi')$ in $\cN_{\Lg'^*}$. Let $K=\{v\in V|\alpha_\xi(v)=0,\beta_\xi(v,V)=0\}$ and for each $(\lambda',\chi')\in\LN_{C_{n-1}}^{*2}$ that arises in this way, let
\begin{eqnarray*}&&Y_{(\lambda',\chi')}=\{V_1\in\mathbb{P}(K)| \text{ the quadratic form }\alpha_\xi'\text{ on }V_1^\p/V_1\text{ induced by }\alpha_\xi\text{ corresponds to }(\lambda',\chi')\},\\&&
X_{(\lambda',\chi')}=\{F=(V_i)\in\cF_\xi^G|V_1\in Y_{(\lambda',\chi')}\}.\end{eqnarray*} Then the fibers of the morphism $$X_{(\lambda',\chi')}\to Y_{(\lambda',\chi')}, \ F=(V_i)\mapsto V_1,$$ are isomorphic to $\cF_{\xi'}^{G'}$, where $G'=Sp(2n-2)$ and the $G'$-orbit of $\xi'$ is $(\lambda',\chi')$. Now we have that $\dim \cF_\xi^G=\max\dim X_{(\lambda',\chi')}$ and by induction hypothesis that $\dim\cF_{\xi'}^{G'}=(\dim Z_{G'}(\xi')-n+1)/2$. For $\xi$ corresponding to $(\lambda,\chi)$, we have $\dim Z_G(\xi)=\sum_{i\geq 1} (i\lambda_i-\chi(\lambda_i))$ (see \cite{X5}). It is then easy to check that $\dim\cF_{\xi}^{G}=(\dim Z_G(\xi)-n)/2$ using the following lemma (\cite[5.3]{X3}).
\begin{lemma}\label{lem-spop} We have $\dim X_{(\lambda',\chi')}=\dim \cF_\xi^G$ if and only if $(\lambda',\chi')$ and $(\lambda,\chi)$ are related as follows:  $\lambda_j'=\lambda_j$ and $\chi'(\lambda_j')=\chi(\lambda_j)$ for $j\notin\{ i-1,i\}$,
$\lambda_{i-1}'=\lambda_i'=\lambda_i-1$,
 and
$\chi'(\lambda_{i-1}')=\chi'(\lambda_{i}')\in\{\chi(\lambda_i),\chi(\lambda_i)-1\}$
satisfies $[\lambda_i'/2]\leq\chi'(\lambda_i')\leq\lambda_i'$,
$\chi(\lambda_{i+1})\leq\chi'(\lambda_i')\leq\chi(\lambda_{i+1})+\lambda_i-\lambda_{i+1}-1$ (this can happen
if $\lambda_{i-1}=\lambda_{i}>\lambda_{i+1}$). Moreover $\dim Y_{(\lambda,\chi')}=i-1$ if
$\chi'(\lambda_i')=\chi(\lambda_i)$ and $\dim Y_{(\lambda,\chi')}=i-2$ if
$\chi'(\lambda_i')=\chi(\lambda_i)-1$.
\end{lemma}

\subsection{}\label{sec-op-o}Suppose that $G=SO(V,Q)$ with $\dim V=2n+1$ in this subsection. We can identify the variety $\cB^G$ with the set $\mathcal{F}^G$ of all complete flags $0=V_0\subset V_1\subset \cdots\subset V_{2n+1}=V$  such that $V_{2n+1-i}=V_i^\p$ and $Q|_{V_i}=0$ for all $i\leq n$.
\begin{lemma}
Suppose that $\xi\in\cN_{\Lg^*}$. We can identify $\cB_\xi^G$ with the set
$$\cF_\xi^G=\{F=(V_0\subset V_1\subset\cdots\subset V_{2n+1})\in\cF^G\,|\,{\beta_\xi(V_i,V_{2n+2-i})=0}\text{ for all }i\leq n+1\}.$$
\end{lemma}
\begin{proof}
Let $\xi\in\mathfrak{o}(2n+1)^*$, $x\in\mathfrak{o}(2n+1)$ and  $F=(V_0\subset V_1\subset \cdots\subset V_{2n+1})\in\cF^G$. There exists a basis $e_i$, $i\in[-n,n]$, of $V$ such that $\beta( e_i,e_j)=\delta_{i+j,0}+\delta_{i,0}\delta_{j,0}$, $Q(e_i)=\delta_{i,0}$, and $V_k=\spn\{e_i,\ i\in[1,k]\}$ for all $k\leq n$. Assume that $xe_i=\sum_{j}x_{ij}e_j$. We have $x_{ij}+x_{-j,-i}=0$ for all $i,j\in[-n,n]-\{0\}$, $x_{i,-i}=0$ and $x_{0,i}=0$ for all $i\in[-n,n]$,   and
\begin{eqnarray*}&&\xi(x)=\sum_{i,j\in[1,n]}x_{ji}\beta_\xi(e_i,e_{-j})+\sum_{1\leq i<j\leq n}x_{j,-i}\beta_\xi(e_{-j},e_{-i})+\sum_{1\leq i<j\leq n}x_{-i,j}\beta_\xi(e_j,e_i)\\&&\qquad\qquad+\sum_{i\in[1,n]}x_{i,0}\beta_\xi(e_0,e_{-i})+\sum_{i\in[1,n]}x_{-i,0}\beta_\xi(e_0,e_i).\end{eqnarray*}
Let $B=\{g\in G\,|\,g V_i=V_i\}$ be the Borel subgroup corresponding to $F$. For $x\in\Lb$, we have $x_{ij}=0$ for $1\leq i<j\leq n$, $x_{i,-j}=0$ for $i,j\in[1,n]$, and $x_{i,0}=0$ for $i\in[1,n]$. Hence $\xi\in\Ln^*$ if and only if $\beta_\xi(e_i,e_{-j})=0$ for all $1\leq i\leq j\leq n$, $\beta_\xi(e_j,e_i)=0$ for all $1\leq i<j\leq n$, and $\beta_\xi(e_0,e_i)=0$ for all $i\in[1,n]$. The lemma follows.\end{proof}

 Suppose that $\xi\in\cN_\Lg^*$ and that  $cl_G(\xi)=(m;(\lambda,\chi))\in\LN_{B_{n}}^{*2}$. Let $F=(V_i)\in\cF_\xi^G$. Then we have $Q|_{V_1}=0$ and $\beta_\xi(V_1,V)=0$. Let $V'=V_1^\p/V_1$. The non-degenerate quadratic form $Q$ on $V$  induces a non-degenerate quadratic form $Q'$ on $V'$; we write $\Lg'=\mathfrak{o}(V',Q')$. The symplectic form $\beta_\xi$ induces a symplectic form $\beta_{\xi'}$ on $V'$, which corresponds to an orbit $(m';(\lambda',\chi'))$ in $\cN_{\Lg'^*}$. Let $K=\{v\in V|Q(v)=0,\beta_\xi(v,V)=0\}$ and for each $(m';(\lambda',\chi'))\in\LN_{B_{n-1}}^{*2}$ that arises in this way, let
\begin{eqnarray*}&&Y_{m;(\lambda',\chi')}=\{V_1\in\mathbb{P}(K)|\text{ the symplectic form }\beta_\xi'\text{ on }V_1^\p/V_1\text{ induced by }\beta_\xi\\&&\qquad\qquad\quad \text{ corresponds to }(m';(\lambda',\chi'))\},\\&&
X_{m';(\lambda',\chi')}=\{F=(V_i)\in\cF_\xi^G|V_1\in Y_{m';(\lambda',\chi')}\}.\end{eqnarray*}
Then the fibers of the morphism $$X_{m;(\lambda',\chi')}\to Y_{m';(\lambda',\chi')},\  F=(V_i)\mapsto V_1,$$ are isomorphic to $\cF_{\xi'}^{G'}$, where $G'=SO(2n-1)$ and the $G'$-orbit of $\xi'$ is $(m';(\lambda',\chi'))$. For $\xi$ corresponding to $(m;(\lambda,\chi))$ we have $\dim Z_G(\xi)=m+\sum_{i\geq 1} ((i+1)\lambda_i-\chi(\lambda_i))$ (see \cite{X5}). Using the same argument as in the case of $G=Sp(2n)$ and the following lemma (\cite[6.3]{X3}), one can easily check that $\dim\cF_{\xi}^{G}=(\dim Z_G(\xi)-n)/2$.
\begin{lemma}\label{lem-soop} We have $\dim X_{m',(\lambda',\chi')}=\dim \cF_\xi^G$ if and only if $(m';(\lambda',\chi'))$ and $(m;(\lambda,\chi))$ are related as follows.\\[5pt]
\indent (a)  $m'=m-1$,
$\lambda_i'=\lambda_i$ and $\chi'(\lambda_i')=\chi(\lambda_i)$ (this can happen if  $m-1\geq \lambda_{1}-\chi(\lambda_{1})$). We have $\dim Y_{m;(\lambda',\chi')}=0$;\\[5pt]
\indent (b)  $m'=m$,
$\lambda_j'=\lambda_j$ and $\chi'(\lambda_j')=\chi(\lambda_j)$ for $j\notin\{ i-1,i\}$,
$\lambda_{i-1}'= \lambda_i'=\lambda_i-1$,
 and
$\chi'(\lambda_{i-1}')=\chi'(\lambda_{i}')\in\{\chi(\lambda_i),\chi(\lambda_i)-1\}$
satisfies $\lambda_i'/2\leq\chi'(\lambda_i')\leq\lambda_i'$,
$\chi(\lambda_{i+1})\leq\chi'(\lambda_i')\leq\chi(\lambda_{i+1})+\lambda_i-\lambda_{i+1}-1$ (this can happen if $\lambda_{i-1}=\lambda_{i}>\lambda_{i+1}$).
We have $\dim Y_{m;(\lambda',\chi')}=i$ if
$\chi'(\lambda_i')=\chi(\lambda_i)$ and $\dim Y_{m;(\lambda',\chi')}=i-1$ if
$\chi'(\lambda_i')=\chi(\lambda_i)-1$.
\end{lemma}

\section{Induction for nilpotent coadjoint orbits}\label{sec-ind}
Suppose that $p=2$ and $G$ is of type $B$ or $C$ in this section. \subsection{}\label{ssec-jind} In this subsection we recall the notion of $j$-induction in $W^\wedge$ (see \cite{L1} and the references therein). For $\rho\in W^\wedge$, let $b_\rho$ be the smallest integer $i$ such that $\rho$ appears in the $i$th symmetric power of the reflection representation of $W$. For a parabolic subgroup $W_J$ of $W$ and $\rho_1\in W_J^\wedge$, there is a unique $\rho\in W^\wedge$ such that $b_\rho=b_{\rho_1}$ and $\rho$ appears in $\text{Ind}_{W_J}^{W}\rho_1$; we write $\rho=j_{W_J}^W\rho_1$.

Let $\mathfrak{X}=\cup_{n\geq 0}\cP_2(n)$ (where we use the notation that $\cP_2(0)=\{0\}$). Let $W_0=\{1\}$ and for $n\geq 1$ we denote  $W_n$ (resp. $S_n$) a  Weyl group of type $B_n$ or $C_n$ (resp. $A_{n-1}$).  Let $\text{sgn}$ denote the sign character of $S_n$. For each $k\geq 1$, we have (\cite{L1})\\[5pt]
\indent\qquad(a) $\qquad j_{W_n\times S_k}^{W_{n+k}}(\tau\boxtimes\text{sgn})=j_k(\tau)\text{ for }\tau=(\mu)(\nu)\in\cP_2(n)\simeq W_n^\wedge,
$\\[5pt]
 where $j_k:\LX\to\LX,\ (\mu)(\nu)\mapsto(\mu')(\nu'),$  is defined by
$$\mu_i'=\left\{\begin{array}{ll}\mu_i+1&\text{if }i\leq (k+1)/2\\\mu_i&\text{otherwise}\end{array}\right.,\  \nu_i'=\left\{\begin{array}{ll}\nu_i+1&\text{if }i\leq k/2\\\nu_i&\text{otherwise.}\end{array}\right.$$
It is easy to see that $j_k$ is injective and $j_k\circ j_l=j_l\circ j_k$.
\subsection{} Let $L$ be a Levi subgroup of a  parabolic subgroup $P$ of $G$. Denote by $\Lp$ and $\Ll$ the Lie algebra of $P$ and $L$
respectively. Choose a maximal torus $T$ and a Borel subgroup $B$ of $G$ such that $L\supset T\subset B\subset P$. Let $R_G$ and $R_L$ be the root system of $(G,T,B)$ and $(L,T,B\cap L)$ respectively.   We define \begin{eqnarray*}
&&\Lp^*=\{\xi\in\Lg^*\,|\,\xi(\Lg_\alpha)=0\text{ for all }\alpha\in R_G^+\backslash R_L^+\},\ \Ln_\Lp^*=\{\xi\in\Lg^*\,|\,\xi(\Lp)=0\},\\
&&\Ll^*=\{\xi\in\Lg^*\,|\,\xi(\Lg_\alpha)=0\text{ for all }\alpha\in R_G\backslash R_L\},
\end{eqnarray*} where
$\Lg_\alpha\subset\Lg$ denotes the root space corresponding to $\alpha$. Note that $\Lp^*=\Ll^*\oplus\Ln_\Lp^*$.

Let $\rc'$ be an $L$-orbit in $\cN_{\Ll^*}$. Since $\cN_{\Lg^*}$ consists of finitely many $G$-orbits, there exists a unique $G$-orbit $\rc$ in $\cN_{\Lg^*}$
such that $\rc\cap(\rc'+\Ln_\Lp^*)$ is dense in $\rc'+\Ln_\Lp^*$ (note that $\rc'+\Ln_\Lp^*\subset\cN_{\Lg^*}$).
Following \cite{LS1} we say that $\rc$ is obtained by inducing $\rc'$ from $\Ll^*$ to $\Lg^*$ and denote $\rc=\text{Ind}_{\Ll^*,\Lp^*}^{\Lg^*}\rc'$.
\begin{proposition}\label{prop-ind}
Suppose that $\rc'\in\underline{\cN_{\Ll^*}}$ and $\rc=\text{Ind}_{\Ll^*,\Lp^*}^{\Lg^*}\rc'$.
 We have $\gamma_{\Lg^*}(\rc)=j_{W_L}^{W}(\gamma_{\Ll^*}(\rc'))$, where $W_L$ is the Weyl group of $L$.\end{proposition}
The proposition is an analog of \cite[3.5]{LS1}. To prove it one can adapt the proof in \cite[4.1]{Sp3} (another proof in unipotent case is given in \cite{LS1}). We outline the proof here. Suppose that $\gamma_{\Lg^*}(\rc)=\rho\in W^\wedge$ and $\gamma_{\Ll^*}(\rc')=\rho'\in W_L^\wedge$. Let $\xi'\in\rc'$ and $\xi\in(\xi'+\Ln_\Lp^*)\cap\rc$. One can show that $b_\rho=\dim\cB_\xi^G$ and $b_{\rho'}=\dim\cB_{\xi'}^L$ by direct computation (using Proposition \ref{prop-sf} and the known information on $\dim Z_G(\xi)$ and on $b_\rho$) or by adapting the proof in \cite{Sp3}. It is easy to adapt the proof in \cite{LS1} to show that $\dim Z_G(\xi)=\dim Z_{L}(\xi')$ and $\dim\cB_\xi^G=\dim\cB_{\xi'}^L$. It then follows that $b_\rho=b_{\rho'}$. Now let $Y_{\xi,\xi'}=\{g\in G|g^{-1}.\xi\in\xi'+\Ln_\Lp^*\}$ and $I_{\xi,\xi'}$ the set of irreducible components of  $Y_{\xi,\xi'}$ of dimension $\frac{1}{2}(\dim Z_G(\xi)+\dim Z_{G'}(\xi'))+\dim U_P$. Let $A_G(\xi)$ denote the component group of $Z_G(\xi)$. Then $A_G(\xi)\times A_L(\xi')$ acts on $I_{\xi,\xi'}$; we denote the corresponding representation by $\varepsilon_{\xi,\xi'}$. We have $I_{\xi,\xi'}\neq \emptyset$  and $\la\rho',\text{Res}_{W_L}^{W}(\rho)\ra_{W_L}=\la 1,\varepsilon_{\xi,\xi'}\ra_{A_G(\xi)\times A_L(\xi')}\neq 0$ (see \cite{Lu1,X3}). It now follows from the definition of $j$-induction that $\rho=j_{W_L}^W(\rho')$.

Note that it follows from Proposition \ref{prop-ind} and the injectivity of $\gamma_{\Lg^*}$ that $\rc=\text{Ind}_{\Ll^*,\Lp^*}^{\Lg^*}\rc'$ does not depend on the choice of $P\supset L$. Hence we can write $\rc=\text{Ind}_{\Ll^*}^{\Lg^*}\rc'$.

\subsection{}\label{lem-j} Let $\tilde{\Lg}=\mathfrak{sp}(2n+2k)$ (resp. $\mathfrak{o}(2n+2k+1)$) be the Lie algebra of $\tilde{G}=Sp(2n+2k)$ (resp. $SO(2n+2k+1)$). Let $P$ be a parabolic subgroup of $\tilde{G}$ such that $\Ll\cong\mathfrak{sp}(2n)\oplus\mathfrak{gl}(k)$ (resp. $ \mathfrak{o}(2n+1)\oplus\mathfrak{gl}(k)$) for  a Levi subgroup $L$ of $P$. Let $\rc$ be the nilpotent orbit in $\Ll^*$ corresponding to the nilpotent orbit $(\lambda,\chi)\in\LN_{C_{n}}^{*2}$ (resp. $(m;(\lambda,\chi))\in\LN_{B_{n}}^{*2}$) and to the $0$ orbit in $\mathfrak{gl}(k)^*$. Let $(\tilde{\lambda},\tilde{\chi})$ (resp. $(\tilde{m};(\tilde{\lambda},\tilde{\chi}))$) correspond to the orbit $\tilde{\rc}=\text{Ind}_{\Ll^*}^{\tilde{\Lg}^*}\rc$.
We write $(\tilde{\lambda},\tilde{\chi})=j_k(\lambda,\chi)$ (resp. $(\tilde{m};(\tilde{\lambda},\tilde{\chi}))=j_k(m;(\lambda,\chi))$) in view of the following  commutative diagrams (see \ref{ssec-jind} (a) and Proposition \ref{prop-ind}) $$\xymatrix{
\LN^{*2} \ar[d]^{\gamma^*} \ar[r]^{j_k} &\LN^{*2}\ar[d]^{\gamma^*}\\
\LX \ar[r]^{j_k} &\LX}$$
where $\gamma^*$ is the Springer correspondence map. Using this one easily sees that (see also \cite{Spal})\\[5pt]
\indent (a) {\em For every $(\lambda,\chi)\in\LN_C^{*2}$ (resp. $(m;(\lambda,\chi))\in\LN_B^{*2}$), there exists a sequence of integers $l_1,\ldots,l_s$ such that $j_{l_1}\circ\cdots\circ j_{l_s}((\lambda,\chi))$ (resp. $j_{l_1}\circ\cdots\circ j_{l_s}(m;(\lambda,\chi))$) is of the form $j_{k_1}\circ \cdots\circ j_{k_r}(0)$ for some sequence $k_1,\ldots,k_r$.}

Here in the expression $j_{k_1}\circ \cdots\circ j_{k_r}(0)$, $0$ denotes the empty partition, not the zero orbit.

\section{Closure relation among nilpotent coadjoint orbits in type $B$, $C$ in characteristic 2}\label{sec-ord}
\subsection{}Assume that $G$ is of type $B$ or $C$. By identifying $W^\wedge$ with the set of $\cP_2(n)$ we get a partial order on $W^\wedge$ (see Subsection \ref{sec-dualnt}). For $\rc, \rc'\in\underline{\cN_{\Lg^*}}$, we say that $\rc\leq\rc'$ if $\rc$ is contained in the Zariski closure $\overline{\rc'}$ of ${\rc'}$, and that $\rc<\rc'$ if  $\rc\leq\rc'$ and $\rc\neq\rc'$. We have
\begin{theorem}\label{mainthm}
Suppose that $p=2$ and $\rc,\rc'\in\underline{\cN_{\Lg^*}}$. We have $\rc<\rc'$ if and only if $\gamma_{\Lg^*}(\rc)<\gamma_{\Lg^*}(\rc')$.
\end{theorem}
Note that theorem is true when $G$ is of type $D$ by \cite{Spal} or when $p\neq 2$ (see \cite{X2}) (in these cases we can identify $\Lg^*$ with $\Lg$). We prove the theorem in the remainder of this section (where we assume that $p=2$) using similar  arguments as in \cite{Spal} most of the time. In the reduction process of \ref{ssec-o1} we use a different argument without using the theory of sheets and packets.

\subsection{}Assume that $G=Sp(V,\la,\ra)$ (resp. $SO(V,Q)$). Let $V_1$ and $V_2$ be orthogonal subspaces of $V$ such that $V=V_1\oplus V_2$. Then the restrictions of $\la,\ra$ (resp. $Q$) on $V_1$ and $V_2$ are non-degenerate. Let $G_1=Sp(V_1,\la,\ra|_{V_1})$ (resp. $G_1=SO(V_1,Q|_{V_1})$) and $G_2=Sp(V_2,\la,\ra|_{V_2})$ (resp. $G_2=SO(V_2,Q|_{V_2})$). Using the isomorphism $\Lg^*\xrightarrow{\sim}\mathfrak{Q}(V)$ (resp. $\Lg^*\xrightarrow{\sim}\mathfrak{S}(V)$) (see \ref{ssec-dualorbits} (a)) we have a natural inclusion $\Lg_1^*\oplus\Lg_2^*\subset\Lg^*$. 
\begin{lemma}\label{lem-r1}
If $\xi_1,\eta_1\in\cN_{\Lg_1^*}$ and $\xi_2,\eta_2\in\cN_{\Lg_2^*}$ are such that $\cl_{G_1}(\xi_1)\leq\cl_{G_1}(\eta_1)$ and $\cl_{G_2}(\xi_2)\leq\cl_{G_2}(\eta_2)$, then $\cl_G(\xi_1+\xi_2)\leq\cl_G(\eta_1+\eta_2)$.
\end{lemma}

\subsection{}\label{ssec-o1}For $\tau\in\LX_R^{*2}$, we denote $\rc_\tau\in\LN_R^{*2}$ the corresponding orbit, where $R$ stands for $B$ or $C$ (type of $G$). We show  by induction on $\dim G$ that if $\tau'\leq\tau$ then $\rc_{\tau'}\leq\rc_\tau$ (see Subsection \ref{ssec-redsp} for $G=Sp(2n)$ and  Subsection \ref{ssec-redo} for $G=SO(2n+1)$).  We may assume that this is true for classical groups of strictly smaller dimension and that $\tau'<\tau$, $\{\tau''\in\LX_R^{*2}|\tau'<\tau''<\tau\}=\emptyset$. We have the following reduction process.

{\em Reduction 1.} Suppose in the case of $G=SO(2n+1)$ (resp. $Sp(2n)$), we have $\Lg^*\supset \Lg^*_1\oplus\Lg_2^*$ with $0\neq\Lg_1^*\neq\Lg^*$ and $\Lg_1^*\cong\Lo(2k+1)^*$, $\Lg_2^*\cong\Lo(2n-2k)^*$ (resp. $\Lg_1^*\cong\mathfrak{sp}(2k)^*$, $\Lg_2^*\cong\mathfrak{sp}(2n-2k)^*$). Assume that we can find $\xi=\xi_1+\xi_2\in\rc_\tau$, $\xi'=\xi_1'+\xi_2'\in\rc_{\tau'}$ with $\xi_1,\xi_1'\in\Lg_1^*$ corresponding to $\tau_1,\tau_1'$ in $\LX$ and $\xi_2,\xi_2'\in\Lg_2^*$ corresponding to $\tau_2,\tau_2'$ in $\LX$, such that $\tau_1'\leq\tau_1$ and $\tau_2'\leq\tau_2$. Then by induction hypothesis and Lemma \ref{lem-r1} we have $\rc_{\tau'}\leq\rc_\tau$.

{\em Reduction 2.} Assume that we can find $\omega>\omega'$ in $\LX_R^{*2}$ such that $\tau=j_{i_1}\circ \cdots\circ j_{i_r}(\omega)$ and $\tau'=j_{i_1}\circ \cdots\circ j_{i_r}(\omega')$ for some positive integers $i_1,\ldots,i_r$. Then by induction hypothesis and the fact that induction for orbits preserves order, we have $\rc_{\tau'}\leq\rc_\tau$.

\subsection{}\label{ssec-redsp}
Assume that $G=Sp(2n)$. Recall that $\LX_{C}^{*2}=\{(\mu)(\nu)|\nu_i\leq\mu_i+1\}$. Let $\tau=(\mu)(\nu),\tau'=(\mu')(\nu')\in\LX_{C_n}^{*2}$ be such that $\tau>\tau'$ and $\{\tau''\in\LX_{C_n}^{*2}|\tau>\tau''>\tau'\}=\emptyset$.
\begin{lemma}\label{lem-od1}
Suppose that reduction 1 does not apply. One of the following is true:\\[5pt]
\indent (a) $\tau=(\mu_1)(\nu_1)$,\  $\tau'=(\mu_1-1)(\nu_1+1)$;\\[5pt]
\indent (b) $\tau=(\mu_1,\mu_2)(\nu_1,\nu_2)$,\  $\tau'=(\mu_1,\mu_2+1)(\nu_1-1,\nu_2)$.
\end{lemma}
\begin{proof} Note that we can apply Reduction 1 if $\mu_1^*\geq 2$ or $\nu_1^*\geq 2$, and  for some $j$, $\mu_j+\nu_j=\mu'_j+\nu'_j$ with $\mu_j\geq\mu_j'$ or $\nu_j'=\nu_k'$ for some $k<j$.   We denote by $(1^a)$  the partition with all parts $1$ and multiplicity $a$. Since Reduction 1 does not apply, if $\mu_1+\nu_1=\mu_1'+\nu_1'$, then we have $\mu_1^*\leq1,\nu_1^*\leq 1$  and thus $\tau,\tau'$ are as in case (a). From now on we assume that $\mu_1+\nu_1>\mu_1'+\nu_1'$.
Let $r=\mu_1$ and $a=\mu_r^*$. Then $r>0$ (otherwise $\tau=(0)(1^n)$ is minimal). We have the following cases.

1) $\mu_1>\mu_1'$. Let $b=\nu_{r}^*$ and $c=\nu_{r+1}^*$. Then $ c\leq a$ and $b\leq\mu_{r-1}^*$.\\[5pt]
 i) If $b<a$, then there exists $\tau''\in\LX_C^{*2}$ such that $\tau=\tau''+(1^a)(1^b)$. One easily verifies that $\tau>\tau''+(1^b)(1^a)\geq\tau'$ (note that $\mu_i-1=\mu_1-1\geq\mu'_i$ for $i\in[b+1,a]$). Hence $\tau'=\tau''+(1^b)(1^a)$.\\[5pt]
ii) If $b\geq a$ and $c\neq 0$, then there exists $\tau''\in\LX_C^{*2}$ such that $\tau=\tau''+(1^a)(1^c)$. One easily verifies that $\tau>\tau''+(1^{a+1})(1^{c-1})\geq\tau'$ (note that $\nu_c-1=\mu_1\geq\nu_c'$ and $\nu_i+\mu_{i+1}-1=2\mu_1-1\geq\nu_i'+\mu_{i+1}'$ for $i\in[c+1,a-1]$). Hence $\tau'=\tau''+(1^{a+1})(1^{c-1})$.\\[5pt]
 iii) If $b\geq a$, $c=0$ and $b<\mu_{r-1}^*$, then there exists $\tau''\in\LX_C^{*2}$ such that $\tau=\tau''+(1^a)(1^b)$. One easily verifies that $\tau>\tau''+(1^{b+1})(1^{a-1})\geq\tau'$ (note that $\mu_a+\nu_a-1=2\mu_1-1\geq\mu_a'+\nu_a'$ and $\nu_i+\mu_{i}=2\mu_1-1\geq\nu_i'+\mu_{i}'$ for $i\in[a+1,b]$). Hence $\tau'=\tau''+(1^{b+1})(1^{a-1})$.\\[5pt]
iv)  If $b\geq a$, $c=0$ and $b=\mu_{r-1}^*$, then let $\tau''=(\mu'')(\nu'')$ with $\nu''=\nu$ and $\mu_i''=\mu_i$ except that $\mu_{a}''=\mu_a-1$ and $\mu''_{b+1}=\mu_{b+1}+1$. Then one easily verifies that $\tau>\tau''\geq\tau'$ (note that $\mu_{a}-1=\mu_1-1\geq\mu_a',\ \mu_a+\nu_a-1=2\mu_1-1\geq\mu_a'+\nu_a'$ and $\mu_i\geq\mu_i',\nu_i\geq\nu_i'$ for $i\in[a+1,b]$). Hence $\tau'=\tau''$.

Since $\mu_1>\mu_1'$, in case (i) we have $b=0$, cases (ii) and (iii) do not happen, in case (iv) we have $a=1$. As Reduction 1 does not apply, in case (i) we have $a=1$,  $\tau=(\mu_1)(\nu_1)$ and $\tau'=(\mu_1-1)(\nu_1+1)$;  in case (iv) we have $b=1$, $\tau=(\mu_1,\mu_2)(\mu_1,\nu_2)$ and $\tau'=(\mu_1-1,\mu_2+1)(\mu_1,\nu_2)$, but notice that $\tau>(\mu_1,\mu_2+1)(\mu_1-1,\nu_2)>\tau'$ ($\nu_2\leq\mu_1-1$ since $b=1$) so case (iv) does not happen. It follows that $\tau$ and $\tau'$ are as in (a).

2) $\mu_1=\mu_1'$. Then $\nu_1>\nu_1'$. Let $s=\nu_1$ and $b=\nu_s^*$. \\[5pt]
i) If $a\leq b$,  then there exists $\tau''\in\LX_C^{*2}$ such that $\tau=\tau''+(1^a)(1^b)$. One easily verifies that $\tau>\tau''+(1^{b+1})(1^{a-1})\geq\tau'$ (note that $\nu_i-1=\nu_1-1\geq\nu_i'$ for $i\in[a,b]$). Hence $\tau'=\tau''+(1^{b+1})(1^{a-1})$.\\[5pt]
ii) If $a>b$ and $\nu_a=s-1$,  then there exists $\tau''\in\LX_C^{*2}$ such that $\tau=\tau''+(1^a)(1^b)$. One easily verifies that $\tau>\tau''+(1^{a+1})(1^{b-1})\geq\tau'$ (note that $\nu_b-1=\nu_1-1\geq\nu_b'$,  $\nu_i=\nu_1-1\geq\nu_i'$ and $\mu_i=\mu_1\geq\mu_i'$ for $i\in[b+1,a]$). Hence $\tau'=\tau''+(1^{a+1})(1^{b-1})$.\\[5pt]
iii) If $a>b$, $\nu_a\leq s-2$,  then there exists $b'\in[b+1,a-1]$ such that $\nu_{b+1}=\cdots=\nu_{b'}=s-1>\nu_{b'+1}$.
Let $\tau''=(\mu'')(\nu'')$ with $\mu''=\mu$ and $\nu_i''=\nu_i$ except that $\nu_{b}''=\nu_b-1$ and $\nu''_{b'+1}=\nu_{b'+1}+1$. Then one easily verifies that $\tau>\tau''\geq\tau'$ (note that $\nu_{b}-1=\nu_1-1\geq\nu_b'$, and $\mu_i=\mu_1\geq\mu_i',\nu_i=\nu_1-1\geq\nu_i'$ for $i\in[b+1,b']$). Hence $\tau'=\tau''$.

Since $\nu_1>\nu_1'$, we have that in case (i) $a=1$, in cases (ii) and (iii) $b=1$. As Reduction 1 does not apply, in case (i) we have $b=1$, $\tau=(\mu_1,\mu_2)(\nu_1,\nu_2)$ and $\tau'=(\mu_1,\mu_2+1)(\nu_1-1,\nu_2)$; case (ii) does not happen ($a\geq2$); in case (iii) we have $\nu_2\leq\nu_1-2$, $\tau=(\mu_1,\mu_1)(\nu_1,\nu_2)$ and $\tau'=(\mu_1,\mu_1)(\nu_1-1,\nu_2+1)$ (but notice that $\tau>(\mu_1,\mu_1-1)(\nu_1,\nu_2+1)>\tau'$, so this case does not happen). It follows that $\tau$ and $\tau'$ are as in (b).
\end{proof}

\begin{proposition}\label{prop-od1}
Suppose that Reductions 1 and 2 do not apply. One of the following is true:
\begin{enumerate}
\item[(i)] $\tau=(1)(0)$, $\tau'=(0)(1)$;
\item[(ii)] $\tau=(1)(1)$, $\tau'=(1,1)(0)$;
\item[(iii)] $\tau=(1)(2)$, $\tau'=(1,1)(1)$;
\item[(iv)] $\tau=(1)(2,1)$, $\tau'=(1,1)(1,1)$;
\item[(v)] $\tau=(\frac{n}{2},\frac{n}{2}-1)(1)$, $\tau'=(\frac{n}{2},\frac{n}{2})(0)$\ \ (for $n\geq 4$ even).
\end{enumerate}
\end{proposition}
\begin{proof}In case (a) of Lemma \ref{lem-od1} we have $\tau=j_1^{\mu_1-\nu_1-1}\circ j_2^{\nu_1}(\omega)$ and $\tau'=j_1^{\mu_1-\nu_1-1}\circ j_2^{\nu_1}(\omega')$, where $\omega=(1)(0)>\omega'=(0)(1)$. In case (b) of Lemma \ref{lem-od1} we have $\tau=j_1^{\mu_1-\mu_2-1}\circ j_3^{\nu_1-\nu_2-1}\circ j_4^{\nu_2}(\omega)$ and $\tau'=j_1^{\mu_1-\mu_2-1}\circ j_3^{\nu_1-\nu_2-1}\circ j_4^{\nu_2}(\omega')$ if $\nu_1\leq\mu_2+1$, where $\omega=(\mu_2-\nu_1+2,\mu_2-\nu_1+1)(1)>\omega'=(\mu_2-\nu_1+2,\mu_2-\nu_1+2)(0)$; $\tau=j_1^{\mu_1-\nu_1}\circ j_2^{\nu_1-\mu_2-1}\circ j_3^{\mu_2-\nu_2}\circ j_4^{\nu_2}(\omega)$ and $\tau'=j_1^{\mu_1-\nu_1}\circ j_2^{\nu_1-\mu_2-1}\circ j_3^{\mu_2-\nu_2}\circ j_4^{\nu_2}(\omega')$ if $\mu_2+2\leq\nu_1\leq\mu_1$ and $\nu_2\leq\mu_2$, where $\omega=(1)(1)>\omega'=(1,1)(0)$; $\tau= j_2^{\nu_1-\mu_2-2}\circ j_3^{\mu_2-\nu_2}\circ j_4^{\nu_2}(\omega)$ and $\tau'= j_2^{\nu_1-\mu_2-2}\circ j_3^{\mu_2-\nu_2}\circ j_4^{\nu_2}(\omega')$ if $\nu_1=\mu_1+1$ and $\nu_2\leq\mu_2$, where $\omega=(1)(2)>\omega'=(1,1)(1)$;  $\tau= j_1^{\mu_1-\nu_1+1}\circ j_2^{\nu_1-\mu_2-2}\circ j_4^{\mu_2}(\omega)$ and $\tau'= j_1^{\mu_1-\nu_1+1}\circ j_2^{\nu_1-\mu_2-2}\circ j_4^{\mu_2}(\omega')$ if $\nu_2=\mu_2+1$ and $\nu_1\geq\mu_2+2$, where $\omega=(1)(2,1)>\omega'=(1,1)(1,1)$. Since Reduction 2 does not apply, the proposition follows.
\end{proof}
It remains to show that in each case of Proposition \ref{prop-od1} we have $\rc_{\tau'}<\rc_{\tau}$. In case (i) it is obvious that $\rc_{\tau'}<\rc_\tau$ as $\tau'$ corresponds to the $0$ orbit. In each case (ii)-(v), we choose an element $\xi\in\rc_\tau$ and define a family $\{g_t\in Sp(V),\ t\in\tk^*\}$ such that $\lim_{t\to 0}g_t^{-1}.\xi=\xi'\in\rc_{\tau'}$. Then it follows that $\rc_{\tau'}<\rc_{\tau}$. The elements $\xi$ (or equivalently the quadratic form $\alpha_\xi$ associated to $\xi$) and the family $\{g_t\in Sp(V),\ t\in\tk^*\}$ in each case are given as follows. We have $\alpha_{g.\xi}(v)=\alpha_\xi(g^{-1}v)$.

Let $e_i,i\in[-n,n]-\{0\}$ be a basis of $V$ such that $\la e_i,e_j\ra=\delta_{i+j,0}$.\\[5pt]
(ii) $\alpha_\xi(\sum a_ie_i)=a_{-2}^2+a_1a_{-2};\ g_te_1=te_1,\ g_te_2=e_2,\ g_te_{-1}=\frac{1}{t}e_{-1},\ g_te_{-2}=e_{-2}$. Then
$\alpha_{g_t^{-1}.\xi}(\sum a_ie_i)=a_{-2}^2+ta_1a_{-2}.$\\[5pt]
(iii) $\alpha_\xi(\sum a_ie_i)=a_1^2+a_1a_{-2}+a_2a_{-3}$;
$g_te_1=e_1,\ g_te_2=\frac{1}{t}e_2,\ g_te_3=\frac{1}{t}e_3,g_te_{-1}=e_{-1},\ g_te_{-2}=te_{-2},\ g_te_{-3}=te_{-3}$. Then $\alpha_{g_t^{-1}.\xi}(\sum a_ie_i)=a_{1}^2+a_2a_{-3}+ta_1a_{-2}$.\\[5pt]
(iv) $\alpha_\xi(\sum a_ie_i)=a_1a_{-2}+a_2a_{-3}$;
$g_te_1=te_1,\ g_te_2=e_2+e_4,\ g_te_3=e_3,\ g_te_4=te_4,\ g_te_{-1}=\frac{1}{t}e_{-1},\ g_te_{-2}=e_{-2},\ g_te_{-3}=e_{-3},\ g_te_{-4}=\frac{1}{t}(e_{-4}+e_{-2})$. Then  $\alpha_{g_t^{-1}.\xi}(\sum a_ie_i)=a_{1}a_{-4}+a_2a_{-3}+ta_1a_{-2}$.\\[5pt]
(v) $\alpha_\xi(\sum a_ie_i)=\sum_{i\in[1,d]}a_ia_{-i-1}+\sum_{i\in[d+2,2d-1]}a_ia_{-i-1}+a_{d}^2+a_{2d}^2$, where $d=\frac{n}{2}$;
$g_t e_1=e_1,\ g_t e_{-1}=e_{-1},\ g_t e_i=\frac{1}{t}(e_i+e_{i+d})\text{ and } g_t e_{-i}=t e_{-i}\text{ for }i\in[2,d],\
g_t e_{d+1}=\frac{1}{t}e_{d+1},\ g_t e_{-d-1}=t e_{-d-1},\ g_t e_i=e_i\text{ and }g_t e_{-i}=e_{-i}+e_{-i+d}\text{ for } i\in[d+2,2d].
$
Then $\alpha_{g_t^{-1}.\xi}(\sum a_ie_i)=a_1a_{-d-2}+\sum_{i\in[2,d]}a_ia_{-i-1}+\sum_{i\in[d+2,2d-1]}a_ia_{-i-1}+a_{2d}^2+ta_1a_{-2}$.

\subsection{}\label{ssec-redo} Assume that $G=SO(2n+1)$. Recall that $\LX_B^{*2}=\{(\mu)(\nu)|\nu_i\geq\mu_{i+1}\}$. Let $\tau=(\mu)(\nu),\tau'=(\mu')(\nu')\in\LX_{B_n}^{*2}$ be such that $\tau>\tau'$ and $\{\tau''\in\LX_{B_n}^{*2}|\tau>\tau''>\tau'\}=\emptyset$.

\begin{proposition}
Suppose that reduction 1 and 2 do not apply. Then $\tau=(1)(n-1)$, $\tau'=(0)(n)$.
\end{proposition}
\begin{proof}
Note that we can apply Reduction 1 if $\mu_1=\mu_1'$ or if for some $j\geq 1$, $\mu_{j+1}+\nu_j=\mu'_{j+1}+\nu'_j$ with $\nu_j\geq\nu_j'$ or $\mu_{j+1}'=\mu_k'$ for some $k<j+1$.

Since Reduction 1 does not apply, we have $\mu_1>\mu_1'$. Let $r=\mu_1$ and $a=\mu_1^*$. Let $\tau''=(\mu'')(\nu'')$ with $\mu''_i=\mu_i-1$ for $i\in[1,a]$ and $\nu_i''=\nu_i+1$ for $i\in[1,a]$. Then one easily verifies that $\tau''\in\LX^{*2}_B$ and $\tau>\tau''\geq\tau'$. Hence $\tau'=\tau''$. Since Reduction 1 does not apply, we have $a=1$, $\tau=(\mu_1,\mu_2)(\nu_1)$ and $\tau'=(\mu_1-1,\mu_2)(\nu_1+1)$. We have $\tau=j_1^{\mu_1-\nu_1-1}\circ j_{2}^{\nu_1-\mu_2}\circ j_3^{\mu_2}((1)(0))$ and $\tau'=j_1^{\mu_1-\nu_1-1}\circ j_{2}^{\nu_1-\mu_2}\circ j_3^{\mu_2}((0)(1))$ if $\mu_1\geq\nu_1+2$; $\tau=j_{2}^{\mu_1-\mu_2-1}\circ j_3^{\mu_2}((1)(\nu_1+1-\mu_1))$ and $\tau'=j_{2}^{\mu_1-\mu_2-1}\circ j_3^{\mu_2}((0)(\nu_1+2-\mu_1))$ if $\mu_1\leq\nu_1+1$. Since Reduction 2 does not apply, $\tau$ and $\tau'$ are as in the proposition.
\end{proof}

Let $\tau=(1)(n-1)$ and $\tau=(0)(n)$. We show that $\rc_\tau>\rc_{\tau'}$. If $n=1$, this is obvious. Assume that $n\geq 2$. Let $e_i$, $i\in[-n,n]$ be a basis of $V$ such that $Q(e_i)=\delta_{i,0}$ and $\la e_i,e_j\ra=\delta_{i+j,0}-\delta_{i,0}\delta_{j,0}$. Let $\beta_\xi$ be the symplectic form corresponding to $\xi\in\cN_{\Lg^*}$ such that $\beta_\xi(\sum_{i\in[-n,n]}a_ie_i,\sum_{i\in[-n,n]}b_ie_i)=a_0b_{-1}+b_0a_{-1}+a_nb_{-1}+b_na_{-1}+\sum_{i\in[2,n-1]}(a_ib_{-i-1}+b_ia_{-i-1})$. Then $\xi\in\rc_\tau$.  Let $g_t\in SO(V)$, $t\in\tk^*$ be defined by:
$$g_te_0=e_0,\ g_te_i=\frac{1}{t}e_i,\ g_te_{-i}=te_{-i},\ i\in[1,n].$$
We have $\beta_{g_t^{-1}.\xi}(\sum a_ie_i,\sum b_ie_i)=\beta_\xi(g_t(\sum a_ie_i),g_t(\sum b_ie_i))=t(a_0b_{-1}+b_0a_{-1})+a_nb_{-1}+b_na_{-1}+\sum_{i\in[2,n-1]}(a_ib_{-i-1}+b_ia_{-i-1})$. Thus $\lim_{t\to 0}g_t^{-1}.\xi \in\rc_{\tau'}$ and $\rc_{\tau'}<\rc_{\tau}$.

\subsection{}\label{ssec-o2} We show that if $\rc_{\tau'}\leq\rc_\tau$ then $\tau'\leq\tau$. Since induction for orbits preserves order and $j_k(\tau')\leq j_k(\tau)$ iff $\tau'\leq\tau$, we may use the operation $j_k$ as often as needed. We show that
 \begin{eqnarray*}&& (*)\qquad \sum_{j\in[1,i]}(\mu_j'+\nu_j')\leq \sum_{j\in[1,i]}(\mu_j+\nu_j),\\&&
 (**)\qquad \sum_{j\in[1,i-1]}(\mu_j'+\nu_j')+\mu_i'\leq \sum_{j\in[1,i-1]}(\mu_j+\nu_j)+\mu_i.\end{eqnarray*}
 As pointed out by the referee, for $Sp(2n)$ the condition $(*)$ follows by considering the partitions corresponding to $T_\tau$ and $T_{\tau'}$, which are $(\mu_1+\nu_1,\ldots)$ and $(\mu_1'+\nu_1',\ldots)$.
 
We write $\text{codim}(\tau,\tau')=\text{codim}_{\overline{\rc_\tau}}\rc_{\tau'}$ for $\tau,\tau'\in\LX^{*2}_R$.  For each $l\geq1$, let  $\Gamma_l'=\{(\tau,\tau')\in\LX^{*2}_R\times\LX^{*2}_R,\ \rc_\tau\geq\rc_{\tau'}\text{ and }(*)\text{ fails for }i=l\}$, $\Delta_l'=\{(\tau,\tau')\in\LX^{*2}_R\times\LX^{*2}_R,\ \rc_\tau\geq\rc_{\tau'}\text{ and }(**)\text{ fails for }i=l\}$, $d=\min\{\text{codim}(\tau,\tau')|(\tau,\tau')\in \Gamma_l'\}$, $d'=\min\{\text{codim}(\tau,\tau')|(\tau,\tau')\in \Delta_l'\}$ and
$$\Gamma_l=\{(\tau,\tau')\in \Gamma_l'|\text{codim}(\tau,\tau')=d\},\ \Delta_l=\{(\tau,\tau')\in \Delta_l'|\text{codim}(\tau,\tau')=d'\}.$$
It is enough to show that $\Gamma_l=\emptyset$ and $\Delta_l=\emptyset$.

It follows from the definitions that we have
\begin{lemma}
(a) if $(\tau,\tau')\in \Gamma_l$ (resp. $\Delta_l$) then $(j_k(\tau),j_k(\tau'))\in \Gamma_l$ (resp. $\Delta_l$).

(b) if $\tau,\tau',\tau''\in\LX^{*2}_R$, $\rc_{\tau}\geq\rc_{\tau''}\geq\rc_{\tau'}$ and $(\tau,\tau')\in \Gamma_l$ (or $\Delta_l$), then $\tau=\tau''$ or $\tau''=\tau'$.
\end{lemma}

Assume that $G=Sp(V,\la,\ra)$ (resp. $SO(V,Q)$). Let $(\tau,\tau')\in \Gamma_l$ or $\Delta_l$. We can assume that $\tau=j_k(\omega)$ for some $\omega\in\LX^{*2}_R$. Let $\Sigma$ be an isotropic subspace of $V$ with $\dim\Sigma=k$ and let $W\subset V$ be such that $\Sigma^\p=\Sigma\oplus W$. Let $P$ be the parabolic subgroup that stabilizes $\Sigma$ and $L$ the Levi subgroup of $P$ that stabilizes $W$. Then $L\cong Sp(2n-2k)\times GL(k)$ (resp. $SO(2n-2k+1)\times GL(k)$). We have $\overline{\rc_{\tau}}=G.(\overline{\rc_\omega}+\Ln_\Lp^*)$.
 Since $\rc_\tau\geq\rc_{\tau'}$, there exists $\xi'\in\rc_{\tau'}$ such that $\alpha_{\xi'}|_{\Sigma}=0$, $\beta_{\xi'}(\Sigma,\Sigma^\p)=0$ (resp. $\beta_{\xi'}(\Sigma,\Sigma^\p)=0$) and such that $\alpha_{\xi'}$ (resp. $\beta_{\xi'}$) induces a quadratic form (resp. symplectic form) on $\Sigma^\p/\Sigma\cong W$ which corresponds to an element  $\eta'\in\overline{\rc_\omega}$. Let $\omega'\in\LX^{*2}_R$ correspond to the $Sp(2n-2k)$-orbit (resp. $SO(2n-2k+1)$-orbit) of $\eta'$. We have $\rc_{\tau'}\leq\text{Ind}_{\Ll^*}^{\Lg^*}\rc_{\omega'}\leq\rc_\tau=\text{Ind}_{\Ll^*}^{\Lg^*}\rc_{\omega}$ and thus\\[5pt]
\indent\qquad either a) $\rc_{\tau'}=\text{Ind}_{\Ll^*}^{\Lg^*}\rc_{\omega'}$\\
\indent\qquad  or b) $\rc_{\omega'}=\rc_{\omega}$.
\\[5pt]
 In case a) we have $\tau'=j_k(\omega')$ and $(\omega,\omega')\in \Gamma_l$ or $\Delta_l$.  Applying $j_k$ when needed we can assume that $\tau=j_{k_r}\circ\cdots\circ j_{k_1}(0)$ for some sequence $k_1,\ldots,k_r$ (see  \ref{lem-j} (a)) and moreover:\\[5pt]
 \indent (i) $k_1=2s$ is even;\\[5pt]
 \indent (ii) $0\leq k_i-k_{i+1}\leq 1$ for $i\leq r-1$ and $k_i=k_{i+1}$ if $3\nmid i$;\\[5pt]
  \indent (iii) ${\mu_1'}^*\leq s,\ {\nu_1'}^*\leq s$.\\[5pt]
We apply the previous construction with $k=k_r$, and repeat the argument  with $(\tau,\tau')$ replaced by $(\omega,\omega')$ (conditions (i)-(iii) above are still satisfied) until we  reach a point that $\omega=\omega'$. Hence we are reduced to the case that $(\tau,\tau')\in \Gamma_l$ (or $\Delta_l$), $\tau=j_{k_r}\circ\cdots\circ j_{k_1}(0)$, $\omega=\omega'=j_{k_{r-1}}\circ\cdots\circ j_{k_1}(0)$ ($\tau,\tau'$ satisfy (i)-(iii) above), where $\tau'$ and $\omega'$ are related as in the above construction. We will study more closely the relation between $\tau'$ and $\omega'$ and derive a contradiction.

Suppose that $\omega=\omega'=(\sigma)(\delta)$. Let $a=[(k_r+1)/2]$, $b=[k_r/2]$.  We show that (see \ref{ssec-spcde} and \ref{ssec-ocde})\\[5pt]
\indent c) $\mu_i'\leq\sigma_1+1=\mu_a$ for all $i$;\\[5pt]
\indent d) $\mu_i'+\nu_i'\leq\sigma_1+\delta_1+2=\mu_b+\nu_b$ for all $i$;\\[5pt]
\indent e)  if $a<s$, then $\mu_i'\geq\sigma_i=\mu_i$ and $\nu_i'\geq\delta_i=\nu_i$ for $i\in[a+1,s]$ (resp. if $b<s$, then $\mu_{i+1}'\geq\sigma_{i+1}=\mu_{i+1}$ and $\nu_i'\geq\delta_i=\nu_i$ for $i\in[b+1,s]$).\\[5pt]
If $(\tau,\tau')\in\Gamma_l$, it follows from c) and d) that $l\geq b+1$ and from e) that $l\leq b$; if $(\tau,\tau')\in\Delta_l$ it follows from c) and d) that $l\geq a+1$  and from e) that $l\leq a$. This is the required contradiction.

Suppose that $r>1$. We have if $k_{r-1}=k_r$ is odd, then $a=b+1$, $\sigma=(r-1)^{a}(3i_2)(3i_4)\cdots$, $\delta=(r-1)^{b}(3i_1)(3i_3)\cdots$; if $k_{r-1}=k_r$ is even, then $a=b$, $\sigma=(r-1)^{a}(3i_1)(3i_3)\cdots$, $\delta=(r-1)^{b}(3i_2)(3i_4)\cdots$; if $k_{r-1}=k_r+1$ is even, then $3|r-1$, $a=b+1$, $\sigma=(r-1)^{a}(3i_1)(3i_3)\cdots$, $\delta=(r-1)^{b+1}(3i_2)(3i_4)\cdots$; if $k_{r-1}=k_r+1$ is odd, then $3|r-1$, $a=b$, $\sigma=(r-1)^{a+1}(3i_2)(3i_4)\cdots$, $\delta=(r-1)^{b}(3i_1)(3i_3)\cdots$; where in each case $r-1>3i_1$ and $i_1>i_2>i_3>\cdots$.

\subsection{}\label{ssec-spcde}We prove \ref{ssec-o2} c)-e) for $G=Sp(V,\la,\ra)$ by studying the quadratic form $\alpha_{\eta'}$ on $\Sigma^\p/\Sigma$ induced by $\alpha_{\xi'}$. Let $c=\max(\mu_1'^*,\nu_1'^*)$.

If $c=1$, then $V$ is indecomposable (as $T_{\xi'}$-module), and $\alpha_{\xi'}= {}^*W_{\chi(\lambda_1)}(\lambda_1)$ (where $\chi(\lambda_1)=\mu_1'$, $\lambda_1=\mu_1'+\nu_1'$; see \ref{lem-sp2} (a)). Then one easily verifies that $(\Sigma^\p/\Sigma,\alpha_{\eta'})\cong {}^*W_{\chi(\lambda_1)-1}(\lambda_1-2)$ (if $\dim\Sigma=2$), ${}^*W_{\chi(\lambda_1)-1}(\lambda_1-1)$ or ${}^*W_{\chi(\lambda_1)}(\lambda_1-1)$ (if $\dim\Sigma=1$). Thus ${\sigma_1}^*\leq 1,{\delta_1}^*\leq 1$ and $\mu_1'-1\leq\sigma_1\leq\mu_1',\ \nu_1'-1\leq \delta_1\leq\nu_1'$; in this case (c) and (d) hold.

If $c>1$, then $V$ is decomposable, and  there exists an orthogonal decomposition  $V=\oplus_{i=1}^c W_i$ with $W_i$ indecomposable, $\dim W_1\geq\dim W_2\geq\cdots\geq\dim W_c$, such that $\Sigma=\oplus_{i=1}^c (W_i\cap\Sigma)$. We apply the previous result for each factor in the decomposition $\Sigma^\p/\Sigma=\oplus_{i=1}^c(W_{i}\cap\Sigma^\p)/(W_i\cap\Sigma)$.
Assume that $\alpha_{\xi'}|_{W_i}= {}^*W_{l_i'}(\lambda_i')$ (where $\lambda_i'=\mu_i'+\nu_i'$). Then $l_i'\leq\mu_i'$  and $l_j'=\mu_1'$ for some $j$. Suppose that  $\alpha_{\eta'}|_{(W_{i}\cap\Sigma^\p)/(W_i\cap\Sigma)}= {}^*W_{l_i''}(\lambda_i'')$. We have $\sigma_1=\max\{l_i''\}\geq\max\{l_i'-1\}=\mu_1'-1$ and $\sigma_1+\delta_1=\max\{\lambda_i''\}\geq \max\{\lambda_i'-2\}=\mu_1'+\nu_1'-2$; thus c) and d) hold.

Now we prove e). Assume that $a<s$. Then $r>1$ and $c=s$. Note that $\sigma_i+\delta_i-(\sigma_{i+1}+\delta_{i+1})\geq 3$ for $i\in[a,s-1]$; hence if $\psi$ is a permutation such that $\dim (W_{\psi(i)}\cap\Sigma^\p)/(W_{\psi(i)}\cap\Sigma)\geq\dim (W_{\psi(i+1)}\cap\Sigma^\p)/(W_{\psi(i+1)}\cap\Sigma)$, then $\psi(i)=i$ for $i\in[a+1,s]$. Moreover  since $\sigma_{i+1}-\sigma_{i+2}\geq 3$ and $\delta_i-\delta_{i+1}\geq 3$ for $i\in[a,s-1]$, $l_i''=\sigma_i$ for all $i\in[a+1,s]$. It follows that $l_i''=r-1=\sigma_i$ for $i\in[1,a]$. Now  we have $l_{i}'\geq l_i''=\sigma_{i}>\sigma_{i+1}+1=l_{i+1}''+1\geq l_{i+1}'$ and  $\lambda_{i}'-l_{i}'\geq\delta_{i}>\delta_{i+1}+1\geq \lambda_{i+1}'-l_{i+1}'$ for $i\in[a+1,s-1]$, $\lambda_a'-l_a'\geq\lambda_a''-l_a''=\delta_j\geq\delta_a>\delta_{a+1}+1\geq \lambda_{a+1}'-l_{a+1}'$ for some $j\in[1,a]$; hence $l_i'=\mu_i'$ for $i\in[a+1,s]$. Then  e) follows.

\subsection{}\label{ssec-ocde}We prove \ref{ssec-o2} c)-e) for $G=SO(V,Q)$ by studying the symplectic form $\beta_{\eta'}$ on $\Sigma^\p/\Sigma$ induced by $\beta_{\xi'}$.  Let $c=\max(\mu_1'^*-1,\nu_1'^*)$.

If $c=0$, then $\mu_1'=m,\ \nu_1'=0$, $V=\spn\{v_i,i\in[0,m],u_i,i\in[0,m-1]\}$ (where $v_i,u_i$ are as in \ref{ssec-dualorbits} (b) (c)), $\Sigma=\spn\{v_0\}$ and  thus $\sigma_1^*\leq 1$, $\sigma_1=m-1,\delta_1=0$; in this case (c) and (d) hold.

If $c>0$, then there exists an orthogonal decomposition $V=\spn\{v_i,i\in[0,m],u_i,i\in[0,m-1]\}\oplus _{i=1}^c W_i$ with $W_i$ indecomposable (as $T_{\xi'}$-module), $\dim W_1\geq\dim W_2\geq\cdots\geq\dim W_s$, such that $\Sigma=(\spn\{v_i,i\in[0,m],u_i,i\in[0,m-1]\}\cap \Sigma)\oplus_{i=1}^c (W_i\cap\Sigma)$. Assume that $T_{\xi'}|_{W_i}=W_{\l_i'}(\lambda'_{i})$ (where $\lambda_{i}'=\mu_{i+1}'+\nu_i'$; notation as in \cite[5.6]{X2}). One easily verifies that $T_{\eta'}|_{(W_i\cap\Sigma^\p)/(W_i\cap\Sigma)}= W_{l_i'-1}(\lambda_{i}'-2)$ (if $\dim(W_i\cap\Sigma)=2$), $W_{l_i'-1}(\lambda_{i}'-1)$ or $W_{l_i'}(\lambda_{i}'-1)$ (if $\dim(W_i\cap\Sigma)=1$).  Write $T_{\eta'}|_{(W_i\cap\Sigma^\p)/(W_i\cap\Sigma)}=W_{l_i''}(\lambda_i'')$. Then $l_i'-1\leq l_i''\leq l_i'$ and $\lambda_i''\geq\lambda'_i-2$. We have $\mu_1'-1\leq \sigma_1\leq\mu_1'$, and $\delta_1=\max\{l_i'',\lambda_1''-\sigma_1\}\geq\max\{l_i'-1,\mu_2'+\nu_1'-2-\sigma_1\}$; thus $\delta_1\geq\nu_1'-2$ if $\sigma_1=\mu_1'$ and $\delta_1\geq\nu_1'-1$  if $\sigma_1=\mu_1'-1$ (note that $\nu_1'=\max\{l_i',\mu_2'+\nu_1'-\mu_1'\}$). Hence c) and d) follow.

We prove e). Assume that $b<s$. Then $r>1$ and $c=s$. Note that $\sigma_{i+1}+\delta_i-(\sigma_{i+2}+\delta_{i+1})\geq 3$ for $i\in[b,s-1]$; hence
 if $\psi$ is a permutation such that $\dim (W_{\psi(i)}\cap\Sigma^\p)/(W_{\psi(i)}\cap\Sigma)\geq\dim (W_{\psi(i+1)}\cap\Sigma^\p)/(W_{\psi(i+1)}\cap\Sigma)$, then $\psi(i)=i$ for $i\in[b+1,s]$. Moreover  for all $i\in[b+1,s]$, $l_i''=\delta_i$  since $\sigma_{i}-\sigma_{i+1}\geq 3$ and $\delta_{i}-\delta_{i+1}\geq 3$. It follows that $l_i''=r-1=\delta_i$ for all $i\in[1,b]$. Now  we have $l_{i}'\geq l_i''=\delta_{i}>\delta_{i+1}+1=l_{i+1}''+1\geq l_{i+1}'$ and $\lambda_{i}'-l_{i}'\geq\sigma_{i+1}>\sigma_{i+2}+1\geq \lambda_{i+1}'-l_{i+1}'$ for $i\in[b+1,s-1]$, $\lambda_b'-l_b'\geq\lambda_b''-l_b''=\sigma_{j+1}\geq\sigma_{b+1}>\sigma_{b+2}+1\geq \lambda_{b+1}'-l_{b+1}'$ for some $j\in[1,b]$ or (if $b=0$) $m\geq\sigma_{b+1}>\sigma_{b+2}+1\geq \lambda_{b+1}'-l_{b+1}'$; hence $l_i'=\nu_i'\geq\delta_i$ and $\mu_{i+1}'=\lambda_i'-l_i'\geq\sigma_{i+1}$ for $i\in[b+1,s]$.  This completes the proof of Theorem \ref{mainthm}.

\section{Nilpotent pieces in $\Lg^*$ in type $B$, $C$ in characteristic 2}\label{sec-comdef}

Assume that $p=2$ and $G$ is of type $B$ or $C$ throughout this section unless otherwise stated.
\subsection{}\label{ssec-comb1} Suppose that $G=Sp(2n)$ (resp. $SO(2n+1)$). Let $\LX_R^1\subset\LX$ denote the image of the Springer correspondence map $\gamma_{G_\mathbb{C}}:\underline{\cU_{G_\mathbb{C}}}\to W^\wedge$, where as before $R$ stands for $B$ or $C$ (the type of $G$ or $G_{\mathbb{C}}$).  We have (see \cite{Lu3,X2}) \\[5pt]\indent\qquad$\LX_{C}^1=\{(\mu)(\nu)|\mu_{i+1}-1\leq\nu_i\leq\mu_i+1\}$,  $\LX_{B}^1=\{(\mu)(\nu)\in|\mu_{i+1}\leq\nu_i\leq\mu_i+2\}$.\\[5pt]
Let $\tilde{\tau}\in\LX_R^1$ and let $\rc\in\underline{\cN_{\Lg^*}}$ be such that $\gamma_{\Lg^*}(\rc)=\tilde{\tau}$. Define $\Sigma_{\tilde{\tau}}^{\Lg^*}$
to be the set of all orbits $\rc'\in\underline{\cN_{\Lg^*}}$ such that $\rc'\leq \rc$
and
$\rc'\nleq\rc''$ for any
$\rc''<\rc$  with
$\gamma_{\Lg^*}(\rc'')\in\LX_R^1$. We show that \\[5pt]
\indent \qquad(a) {\em$\{\Sigma_{\tilde{\tau}}^{\Lg^*}\}_{\tilde{\tau}\in\LX_R^1}$ form a partition of  $\cN_{\Lg^*}$.}

Following \cite{X2}, we define maps
$$ \Phi_R: \LX_{R}^{*2}\rightarrow\LX_R^1,\
(\mu)(\nu)\mapsto(\tilde{\mu})(\tilde{\nu}),$$
where if $R=B$, then
\begin{eqnarray*}\label{phi-soodd1}
&&\tilde{\mu}_i=\left\{\begin{array}{ll}[\frac{\mu_i+\nu_{i}-1}{2}]&\text{
if }\nu_i>\mu_i+2\\ \mu_i&\text{ if
}\nu_i\leq\mu_i+2\end{array}\right.,\quad
\tilde{\nu}_i=\left\{\begin{array}{ll}[\frac{\mu_{i}+\nu_{i}+2}{2}]&\text{
if }\nu_i>\mu_{i}+2\\ \nu_i&\text{ if
}\nu_i\leq\mu_{i}+2\end{array}\right.,\ i\geq1;
\end{eqnarray*}
if $R=C$, then $\tilde{\mu}_1= \mu_1$ and
\begin{eqnarray*}\label{phi-sp1}
\tilde{\mu}_{i+1}=\left\{\begin{array}{ll}
{}[\frac{\mu_{i+1}+\nu_{i}+1}{2}]&\text{if }\nu_i<\mu_{i+1}-1\\
\mu_{i+1}& \text{if }\nu_i\geq\mu_{i+1}-1\end{array}\right. ,\ \tilde{\nu}_i=\left\{\begin{array}{lll} {}[\frac{\mu_{i+1}+\nu_{i}}{2}]&\text{if }\nu_i<\mu_{i+1}-1\\
\nu_i&\text{if
}\nu_i\geq\mu_{i+1}-1\end{array}\right.,\ i\geq1.
\end{eqnarray*}
It is easy to verify that in each case we get a well-defined element
$(\tilde{\mu})(\tilde{\nu})\in \LX^1_R$. We have\\[5pt]
\indent(b) {\em if $(\tilde{\mu}')(\tilde{\nu}')\in\LX_R^1$, $(\mu)(\nu)\in\LX^{*2}_R$, then $\Phi_R((\mu')(\nu'))=(\mu')(\nu')$, $(\mu)(\nu)\leq\Phi_R((\mu)(\nu))$; if moreover $(\mu)(\nu)\leq(\tilde{\mu}')(\tilde{\nu}')$, then $\Phi_R((\mu)(\nu))\leq(\tilde{\mu}')(\tilde{\nu}')$.}\\[5pt]
In fact, if $R=C$, (b) follows from \cite[4.2]{X2}; if $R=B$, one can prove (b) by the same argument. Now in view of Theorem \ref{mainthm}, it follows from the definition of ${\Sigma}_{\tilde{\tau}}^{\Lg^*}$ and  (b), (c) that for each $\tilde{\tau}\in\LX_R^1$,\\[5pt]
\indent\qquad(d)  {\em $\gamma_{\Lg^*}({\Sigma}_{\tilde{\tau}}^{\Lg^*})=\Phi_R^{-1}(\tilde{\tau})$.}\\[5pt]
Then (a) follows from (d).

\subsection{}\label{ssec-map} We define  a map $$\Psi_{R}^*:\LN_{R}^{*2}\rightarrow \underline{\cU}_{G_{\mathbb{C}}},\ \rc\mapsto\tilde{\lambda}=(\tilde{\lambda}_1\geq\tilde{\lambda}_2\geq\cdots)$$ as follows such that each fiber is an nilpotent piece (see Proposition \ref{prop-psi}). 

Assume that $G=Sp(2n)$ and $\rc=(\lambda,\chi)\in\LN_{C}^{*2}$. If $\chi(\lambda_{1})>\frac{\lambda_{1}}{2}$, then $\tilde{\lambda}_1=2\chi(\lambda_1)$, if  $\chi(\lambda_{2i})>\frac{\lambda_{2i}}{2}$ and $\chi(\lambda_{2i})>\chi(\lambda_{2i+1})$, then $$\tilde{\lambda}_{2i}=\left\{\begin{array}{lll}\lambda_{2i}-\chi(\lambda_{2i})+\chi(\lambda_{2i+1})&\text{ if }\chi(\lambda_{2i})-\lambda_{2i}+\chi(\lambda_{2i+1})\geq1\\
2(\lambda_{2i}-\chi(\lambda_{2i}))&\text{ if }\chi(\lambda_{2i})-\lambda_{2i}+\chi(\lambda_{2i+1})\leq0\end{array}\right.,$$ if   $\chi(\lambda_{2i+1})>\frac{\lambda_{2i+1}}{2}$ and $\lambda_{2i+1}-\chi(\lambda_{2i+1})<\lambda_{2i}-\chi(\lambda_{2i})$, then $$\tilde{\lambda}_{2i+1}=\left\{\begin{array}{lll}\lambda_{2i}-\chi(\lambda_{2i})+\chi(\lambda_{2i+1})&\text{ if }\chi(\lambda_{2i})-\lambda_{2i}+\chi(\lambda_{2i+1})\geq1\\
2\chi(\lambda_{2i+1})&\text{ if }\chi(\lambda_{2i})-\lambda_{2i}+\chi(\lambda_{2i+1})\leq0\end{array}\right.;$$ otherwise $\tilde{\lambda}_{i}={\lambda}_{i}$.

Assume  that $G=SO(2n+1)$ and $\rc=(m;(\lambda,\chi))\in\LN_{B}^{*2}$. Let
$$\tilde{\lambda}_{1}=\left\{\begin{array}{lll}m+\chi(\lambda_{1})&\text{ if }\chi(\lambda_{1})\geq m+2\\
2m+1&\text{ if }\chi(\lambda_{1})<m+2\end{array}\right.,\ \tilde{\lambda}_{2}=\left\{\begin{array}{lll}m+\chi(\lambda_{1})&\text{ if }\chi(\lambda_{1})\geq m+2\\
2\chi(\lambda_1)-1&\text{ if }[\frac{\lambda+1}{2}]<\chi(\lambda_{1})<m+2\\\lambda_1&\text{if }\chi(\lambda_{1})\leq[\frac{\lambda+1}{2}]\end{array}\right..$$
For $i\geq 1$, if $\chi(\lambda_{2i})>\frac{\lambda_{2i}}{2}$ and $\chi(\lambda_{2i})>\chi(\lambda_{2i+1})$, then $$\tilde{\lambda}_{2i+1}=\left\{\begin{array}{lll}\lambda_{2i}-\chi(\lambda_{2i})+\chi(\lambda_{2i+1})&\text{ if }\chi(\lambda_{2i})-\lambda_{2i}+\chi(\lambda_{2i+1})\geq2\\
2(\lambda_{2i}-\chi(\lambda_{2i}))+1&\text{ if }\chi(\lambda_{2i})-\lambda_{2i}+\chi(\lambda_{2i+1})\leq1\end{array}\right.,$$ if   $\chi(\lambda_{2i+1})>\frac{\lambda_{2i+1}}{2}$ and $\lambda_{2i+1}-\chi(\lambda_{2i+1})<\lambda_{2i}-\chi(\lambda_{2i})$, then $$\tilde{\lambda}_{2i+2}=\left\{\begin{array}{lll}\lambda_{2i}-\chi(\lambda_{2i})+\chi(\lambda_{2i+1})&\text{ if }\chi(\lambda_{2i})-\lambda_{2i}+\chi(\lambda_{2i+1})\geq2\\
2\chi(\lambda_{2i+1})-1&\text{ if }\chi(\lambda_{2i})-\lambda_{2i}+\chi(\lambda_{2i+1})\leq1\end{array}\right.;$$ otherwise $\tilde{\lambda}_{2i+1}={\lambda}_{2i}$, $\tilde{\lambda}_{2i+2}={\lambda}_{2i+1}$.

We show  that\\[5pt]
\indent\qquad (a) {\em $\gamma_{G_{\mathbb{C}}}\circ\Psi_{R}^*=\Phi_R\circ\gamma_{\Lg^*}.$}\\[5pt]
Let  $G=SO(2n+1)$ (resp. $Sp(2n)$) and $\rc\in\LN_{R}^{*2}$. Assume that  $\gamma_{\Lg^*}(\rc)=(\mu)(\nu)$, $\Phi_R((\mu)(\nu))=(\tilde{\mu})(\tilde{\nu})$ and $(\tilde{\mu})(\tilde{\nu})=\gamma_{G_{\mathbb{C}}}(\tilde{\lambda})$.  Using the definition of $\Phi_R$, one easily shows that $\tilde{\nu}_i=\tilde{\mu}_i+2$ iff $\nu_i\geq\mu_i+2$ and $\mu_i+\nu_i$ is even, and $\tilde{\nu}_i=\tilde{\mu}_{i+1}$ iff $\nu_i=\mu_{i+1}$ (resp. $\tilde{\nu}_i=\tilde{\mu}_i+1$ iff $\nu_i=\mu_i+1$, and $\tilde{\nu}_i=\tilde{\mu}_{i+1}-1$ iff $\nu_i\leq\mu_{i+1}-1$ and $\nu_i+\mu_{i+1}$ is odd). Using this and the description of the map $\gamma_{G_{\mathbb{C}}}$ (see \cite[2.4]{X2}), one easily verifies that
 $$\tilde{\lambda}_{2i-1}=\left\{\begin{array}{ll}\mu_i+\nu_i&\text{ if }\mu_i\leq\nu_i-2\\2\mu_i+1&\text{ if }\nu_i-2<\mu_i<\nu_{i-1}\\2\mu_i&\text{ if }\mu_i=\nu_{i-1}\end{array}\right.,\ \tilde{\lambda}_{2i}=\left\{\begin{array}{ll}\mu_i+\nu_i&\text{ if }\nu_i\geq\mu_i+2\\2\nu_{i}-1&\text{ if }\mu_{i+1}<\nu_i<\mu_{i}+2\\2\nu_i&\text{ if }\nu_i=\mu_{i+1}\end{array}\right.$$
  $$\left(\text{resp. }\tilde{\lambda}_{2i-1}=\left\{\begin{array}{ll}\mu_i+\nu_{i-1}&\text{ if }\mu_i\geq\nu_{i-1}+1\\2\mu_i&\text{ if }\nu_i\leq\mu_i\leq\nu_{i-1}\\2\mu_i+1&\text{ if }\mu_i=\nu_i-1\end{array}\right.,\ \tilde{\lambda}_{2i}=\left\{\begin{array}{ll}\mu_{i+1}+\nu_{i}&\text{ if }\nu_i\leq\mu_{i+1}-1\\2\nu_i&\text{ if }\mu_{i+1}\leq\nu_i\leq\mu_{i}\\2\nu_i-1&\text{ if }\nu_i=\mu_i+1\end{array}\right.\right).$$
It is then easy to verify that $\tilde{\lambda}=\Psi_R^*(\rc)$ using the description of $\gamma_{\Lg^*}$ (see \ref{ssec-spcp}).

\begin{proposition}\label{prop-psi}
Two orbits $\rc_1,\rc_2\in \underline{\cN_{\Lg^*}}$  lie in the same   nilpotent piece as defined in \cite{L4,X4} if and only if $\Psi_{R}^*(\rc_1)=\Psi_{R}^*(\rc_2)$.
\end{proposition}
Note that the proposition computes the nilpotent pieces in $\Lg^*$ explicitly. Now in view of  (a) and \ref{ssec-comb1} (d), it follows from Proposition \ref{prop-psi} that each set $\Sigma_{\tilde{\tau}}^{\Lg^*}$ is a nilpotent piece defined in \cite{L4,X4}.
One can also define a partition of $\cN_{\Lg^*}$ into special pieces as in \cite{Lu7,X2} and show that each special piece is a union of  nilpotent pieces. The proof of the proposition is given in the remainder of this section following the argument used in \cite{X2}.

\subsection{}\label{pf-char2}Suppose that $G=Sp(V,\la,\ra)$ (resp. $SO(V,Q)$). Let $\rc\in\underline{\cN_{\Lg^*}}$ and
$\Psi_{R}^*(\rc)=\tilde{\rc}=\tilde{\lambda}$. Suppose that $\Upsilon_{\rc}=(f_a)_{a\in\mathbb{N}}$ and $\Upsilon_{\tilde{\rc}}=(\tilde{f}_a)_{a\in\mathbb{N}}$ (see Subsection \ref{ssec-p}).
We show that\\[5pt]
\indent\qquad (a) {\em $f_a=\tilde{f}_a$ for all ${a\in\mathbb{N}}$}.\\[5pt]
Then Proposition \ref{prop-psi} follows from (a) and Lemma
\ref{lempiece}.

We prove (a) by induction
on $\dim V$. Let  $\xi\in\rc$.
If $\xi=0$, (a) is obvious. Assume  from now on that $\xi\neq0$.
Let $V_*=(V_{\geq a})$,  $V'$, $\xi'$  be associated to $\xi$  and $N$, $e$, $f$ defined for $\xi$ as in
\ref{lem-sp1} (resp. \ref{lem-o1}).   Let $\rc'$ be the orbit of $\xi'$ in  $\Lg'^*$ and let
$\tilde{\rc}'=\Psi_{\Lg'^*}^2(\rc')=\tilde{\lambda}'$.
Suppose that $\Upsilon_{\rc'}=(f_a')$ and $\Upsilon_{\tilde{\rc}'}=(\tilde{f}_a').$  Since $\dim V'<\dim V$, by induction hypothesis
$f_a'=\tilde{f}_a'$ for all $a\in\mathbb{N}$. By the definition of $V_*$ we have that for all $a\in[0,N-1]$, $f_a=f_a'$  and thus $f_a=\tilde{f}'_a$.
We show  that\\[5pt]
\indent\qquad (b) {\em $\begin{array}{l}
\tilde{\lambda}_1=N+1, m_{\tilde{\lambda}}(\tilde{\lambda}_1)=f_N,  m_{\tilde{\lambda}'}(\tilde{\lambda}_1)=0,\
m_{\tilde{\lambda}'}(\tilde{\lambda}_1-2)=m_{\tilde{\lambda}}(\tilde{\lambda}_1-2)+m_{\tilde{\lambda}}(\tilde{\lambda}_1),
\\[5pt]m_{\tilde{\lambda}'}(i)=m_{\tilde{\lambda}}(i) \text{ for all }i\neq \tilde{\lambda}_1,\tilde{\lambda}_1-2.\end{array}$}\\[5pt]
It then follows from (b) and \ref{ssec-p} (a$'$) that $\tilde{f}_a=0$ for all $a\geq N+1$, $\tilde{f}_N=f_N$, and  that $\tilde{f}_a=\tilde{f}_a'$ for all $a\in[0,N-1]$.  Hence (a) follows (note that $f_a=0$ for all $a\geq N+1$).

The proof of (b)  is given in subsections \ref{ssec-sppf1}-\ref{ssec-sppf-3} (resp. \ref{pf-char2-2}-\ref{pf-char2-6}).

\subsection{}\label{ssec-sppf1} Assume that $G=Sp(V,\la,\ra)$ throughout subsection \ref{ssec-sppf-3}. We keep the notations in \ref{pf-char2}. Suppose that $\rc=(\lambda,\chi)$. We show that\\[5pt]
\indent\qquad (a) {\em $
f_N=\left\{\begin{array}{ll}1&\text{ if }e< 2f,\\
m_{{\lambda}}(e)&\text{ if }e=2f+1,\text{ or }e=2f\text{ and }\chi(e-1)=f-1,\\
m_{{\lambda}}(e)+1&\text{if }e=2f\text{ and }\chi(e-1)=f.\end{array}\right.$}\\[5pt]
Recall that $f_N=\dim V_{\geq N} $ and $V_{\geq N}= V_{\geq -N+1}^\p$ (see \ref{lem-sp1}). We describe $V_{\geq N}$ in various cases in the following and then (a) follows.\\[5pt]
\indent Suppose that $e<2f$. Then the map $\rho:V\to\tk,\ v\mapsto\sqrt{\alpha_\xi(T_\xi^{f-1}v)}$ is linear and thus $V_{\geq N}=(\ker\rho)^\p$ is a line.\\[5pt]
\indent Suppose that $e=2f+1$, or $e=2f$ and $\chi(e-1)=f-1$. Then $V_{\geq N}=(\ker T_\xi^{e-1})^\p=\text{Im}T_\xi^{e-1}$.\\[5pt]
\indent Suppose that $e=2f$ and $\chi(e-1)=f$. Let $E$ be subspace of $V$ such that $V=\ker T_\xi^{e-1}\oplus E$ and let $W=\sum_{i\in[0,e]}T_\xi^i E$. Then $\la,\ra|_W$ is non-degenerate (in fact, if $\la\sum_iT_\xi^iv_i, W\ra=0$, where $v_i\in E$, then $\la T_\xi^{e-1}v_0,V\ra=0$ and thus $v_0\in E\cap\ker T_\xi^{e-1}=0$; now use induction and similar argument one shows that $v_i=0$). Thus $V=W\oplus W^\p$, $W^\p$ is $T_\xi$-stable and $T_\xi^{e-1}W^\p=0$ ($W\supset T_\xi^{e-1}V$ implies that $W^\p\subset(\text{Im} T_\xi^{e-1})^\p=\ker T_\xi^{e-1}$). We have $V_{\geq-N+1}=(\ker T_\xi^{e-1}\cap W)\oplus\{v\in W^\p|\alpha_\xi(T_\xi^{f-1}v)=0\}$ (note that $\ker T_\xi^{e-1}\cap W\subset\text{Im}T_\xi$). Thus $V_{\geq N}=\text{Im}T_\xi^{e-1}\oplus L$, where $L\subset W^\p$ is a line (we apply the result in the first case for $W^\p$).

\subsection{}\label{ssec-sppf-2}We describe $\rc'=(\lambda',\chi')$  in various cases as follows.   Let  $j\geq 0$ be the unique integer such that $\chi(e-j)=f\text{ and }\chi(e-j-1)<f.$\\[5pt]
 (i) $e=2f+1$, or $e=2f$ and $\chi(e-1)=f-1$. We have \\[5pt]\indent {\em $m_{\lambda'}(e)=0,\ m_{\lambda'}(e-2)=m_\lambda(e)+m_\lambda({e-2})$ (if $e>2$),$\ m_{\lambda'}(i)=m_{\lambda}(i)\text{ for }i\notin\{e,e-2\}$, $ \chi'(\lambda_i)=\chi(\lambda_i)$ for $\lambda_i\notin\{e, e-2\}$, $\chi'(e-2)=f-1$ if $\chi(e-2)\leq f-1$ and $\chi'(e-2)=f$ if $\chi(e-2)=f$.}\\[5pt]
 (ii) $e=2f$ and $\chi(e-1)=f$. We have \\[5pt]\indent{\em$m_{\lambda'}(e)=0,\ m_{\lambda'}(e-2)=m_{\lambda}(e-2)+m_{\lambda}(e)+2\delta_{j,1}-2\delta_{j,2}$ (if $e>2$), $ m_{\lambda'}(e-j)=m_{\lambda}(e-j)-2+\delta_{j,2}m_\lambda(e)$,$\ m_{\lambda'}(e-j-1)=m_{\lambda}(e-j-1)+2+\delta_{j,1}m_\lambda(e)$ (if $e>j+1$), $ m_{\lambda'}(i)=m_{\lambda}(i)\text{ for }i\notin\{ e,e-2,e-j,e-j-1\}$, $\chi'(e-k)=f-1\text{ for }k\in[1,j],\ \chi'(i)=\chi(i)$ for $i\leq e-j-1$.}\\[5pt]
 (iii) $e<2f$. We have \\[5pt]\indent{\em $m_{\lambda'}(e-j)=m_\lambda(e-j)-2$,$\ m_{\lambda'}(e-j-1)=m_\lambda(e-j-1)+2$  (if $e>j+1$),$\ m_{\lambda'}(i)=m_{\lambda}(i)\text{ for }i\notin\{ e-j,e-j-1\}$, $\chi'(e-k)=f-1\text{ for }k\in[0,j],\ \chi'(i)=\chi(i)$ for $i\leq e-j-1$.}

Let $\alpha_{\xi'}$ and $T_{\xi'}$ be defined for $\xi'$ as before. Let $m_\lambda(e-i)=2d_i$.

Assume that we are in case (i).  We have a decomposition $V=W\oplus Y$  of $V$ into mutually orthogonal $T_\xi$-stable subspaces such that\\[5pt]
\indent\hspace{1.25in} $\alpha_\xi|_W={}^*W_f(e)^{{d_0}} $ and $T_\xi^{e-1}Y=0$.\\[5pt] Then $V_{\geq -N+1}=(\ker T_\xi^{e-1}\cap W)\oplus Y$ and $V_{\geq N}=T_\xi^{e-1}W$. Hence we have a natural decomposition  $V'=W'\oplus Y$ of $V'$ into mutually orthogonal $T_{\xi'}$-stable subspaces, where $W'=(\ker T_\xi^{e-1}\cap W)/T_\xi^{e-1}W$, and  (see \ref{lem-sp2} (ii))\\[5pt]
\indent\hspace{1.25in} $\alpha_{\xi'}|_{W'}={}^*W_{f-1}(e-2)^{{d_0}} ,\ \alpha_{\xi'}|_Y=\alpha_{\xi}|_Y$. \\[5pt]
We have  $\chi'(i)=\max(\chi_{{\alpha_{\xi'}}|_{W'}}(i),\chi_{\alpha_{\xi}|_Y}(i))$ and $\chi(i)=\max(\chi_{\alpha_{\xi}|_W}(i),\chi_{\alpha_{\xi}|_Y}(i))$. For $0<\lambda_i\leq e-3$, $\lambda_i-\chi(\lambda_i)\leq\frac{\lambda_i+1}{2}<e-f$ and thus $\chi_{\alpha_\xi|_W}(\lambda_i)=\max(0,\lambda_i-e+f)<\chi(\lambda_i)$, which implies that $\chi_{\alpha_\xi|_Y}(\lambda_i)=\chi(\lambda_i)$ and thus $\chi'(\lambda_i)=\max(\max(\lambda_i-e+f+1,0),\chi_{\alpha_\xi|_Y}(\lambda_i))=\chi(\lambda_i)$.  Now $\chi'(e-2)=\max(f-1,\chi_{\alpha_{\xi}|_Y}(e-2))$ and $\chi(e-2)=\max(f-2,\chi_{\alpha_{\xi}|_Y}(e-2))$, thus the assertion on $\chi'(e-2)$ holds; $\chi'(e-1)=\chi(e-1)$ since $\chi_{{\alpha_{\xi'}}|_{W'}}(e-1)=\chi_{{\alpha_{\xi}}|_{W}}(e-1)$.

Assume that we are in case (ii).  Then $j\geq 1$.  We have a decomposition $V=W_0\oplus W_1\oplus Y$ of $V$ into mutually orthogonal $T$-stable subspaces such that (we use \cite[Lemma 2.9]{X5})\\[5pt]
\indent \qquad$\alpha_\xi|_{W_0}={}^*W_{f}(e)^{{d_0}} $, $\alpha_\xi|_{W_1}={}^*W_f(e-j)$, $T_\xi^{e-1}Y=0$ and $\alpha_\xi(T_\xi^{f-1}Y)=0$.\\[5pt]
Then $V_{\geq -N+1}=(\ker T_\xi^{e-1}\cap{W_0})\oplus K_{W_1}\oplus Y$ and $V_{\geq N}=T_\xi^{e-1}W_0\oplus L_{W_1},$ where $K_{W_1}=\{v\in W_1|\alpha_\xi(T_\xi^{f-1}v)=0\}$ and $L_{W_1}=K_{W_1}^{\p}\cap W_1$. Hence we have a natural decomposition $V'=W_0'\oplus W_1'\oplus Y$ of $V'$ into mutually orthogonal $T_{\xi'}$-stable subspaces, where $W_0'=(\ker T_\xi^{e-1}\cap{W_0})/T_\xi^{e-1}W_0,\ W_1'=K_{W_1}/L_{W_1}$, and (see \ref{lem-sp2} (i) (ii))\\[5pt]
\indent\qquad$\alpha_{\xi'}|_{W_0'}={}^*W_{f-1}(e-2)^{{d_0}} ,\ \alpha_{\xi'}|_{W_1'}={}^*W_{f-1}(e-j-1),\ \alpha_{\xi'}|_Y=\alpha_\xi|_Y$.\\[5pt]
We have $\chi'(i)=\max(\chi_{\alpha_{\xi'}|_{W_0'}}(i),\chi_{\alpha_{\xi'}|_{W_1'}}(i),\chi_{\alpha_{\xi}|_Y}(i))$ and $\chi(i)=\max(\chi_{\alpha_{\xi}|_{W_0}}(i),\chi_{\alpha_{\xi}|_{W_1}}(i),\chi_{\alpha_{\xi}|_Y}(i))$. Thus for $e-j\leq i\leq e-1$, $\chi'(i)=f-1$ (note that $\chi_{\alpha_{\xi}|_Y}(i)\leq f-1$); for $i\leq e-j-1$, $\chi_{\alpha_{\xi}|_{W_0}}(i)\leq\chi_{\alpha_{\xi'}|_{W_0'}}(i)\leq\chi_{\alpha_{\xi'}|_{W_1'}}(i)=\chi_{\alpha_{\xi}|_{W_1}}(i)$ (since $j\geq 1$) and thus $\chi'(i)=\chi(i)$.

Assume that we are in case (iii).  We have a decomposition of $V$ into mutually orthogonal $T_\xi$-stable subspaces $V=W\oplus Y$  such that (we use  \cite[Lemma 2.9]{X5})\\[5pt]
\indent \hspace{1.25in} $\alpha_{\xi}|_W={}^*W_{f}(e-j)$  and $\alpha_\xi(T_\xi^{f-1}Y)=0$.\\[5pt]
Then $V_{\geq -N+1}=K_W\oplus Y$ and $ V_{\geq N}=L_W,$ where $K_W=\{v\in W|\alpha_\xi(T_\xi^{f-1}v)=0\}$ and $L_W=K_W^{\p}\cap W$. Hence we have a natural decomposition $V'=W'\oplus Y$ of $V'$ into mutually orthogonal $T'$-stable subspaces, where $W'=K_W/L_W$. Moreover (see  \ref{lem-sp2} (i))\\[5pt]
\indent\hspace{1.25in}   $\alpha_{\xi'}|_{W'}={}^*W_{f-1}(e-j-1)$ and $\alpha_{\xi'}|_Y=\alpha_{\xi}|_Y$. \\[5pt]
We have  $\chi'(i)=\max(\chi_{{\alpha_{\xi'}}|_{W'}}(i),\chi_{\alpha_{\xi}|_Y}(i))$ and $\chi(i)=\max(\chi_{\alpha_{\xi}|_W}(i),\chi_{\alpha_{\xi}|_Y}(i))$. Thus for $i\geq e-j$, $\chi'(i)=f-1$   (as $\chi_{\alpha_{\xi}|_Y}(i)\leq f-1$);  for $i\leq e-j-1$, $\chi_{{\alpha_{\xi'}}|_{W'}}(i)=\chi_{\alpha_{\xi}|_W}(i)$ and thus $\chi'(i)=\chi(i)$.

\subsection{}\label{ssec-sppf-3}Using the definition of $\Psi_{R}^*$ and the  description of $\rc'$ in \ref{ssec-sppf-2}, we compute $\Psi_{R}^*(\rc)=\tilde{\lambda}$ and $\Psi_{R'}^*(\rc')=\tilde{\lambda}'$ in each case (i)-(iii) as follows. It is then easy to check that \ref{pf-char2} (b) holds in each case.

Let $d=d_0+d_1$ in case (i) and $d=\sum_{a\in[0,j]}d_a$ in cases (ii) and (iii). We have $\tilde{\lambda}_i=\tilde{\lambda}_i'$ for all $i\geq 2d+2$, since $\mu_i=\mu_i'$  and $\nu_i=\nu_i'$ for all $i\geq d+1$. In case (i), if $\chi(e-2)\leq f-1$, then $\tilde{\lambda}_{2d+1}=\tilde{\lambda}_{2d+1}'$ since $\mu_{d+1}\leq\nu_{d}$ and $\mu_{d+1}'\leq\nu_{d}'$. In cases (ii) and (iii), $\tilde{\lambda}_{i}=\tilde{\lambda}_{i}'$ for $i=2d,2d+1$, since $\nu_{d}=\nu_{d}'$, $\nu_{d}\leq\mu_{d}$ and $\nu_{d}'\leq\mu_d'$. Let $\tilde{\lambda}^1=(\tilde{\lambda}_1,\ldots,\tilde{\lambda}_{2d+1})$ and $\tilde{\lambda}'^1=(\tilde{\lambda}'_1,\ldots,\tilde{\lambda}'_{2d+1})$. We have

\noindent(i)
$\begin{array}{ll}\tilde{\lambda}^1=e^{2d_0}(e-1)^{2d_1+1},\ \tilde{\lambda}'^1=(e-1)^{2d_1+1}(e-2)^{2d_0}&\text{if }\chi(e-2)=f,\\
\tilde{\lambda}^1=e^{2d_0}(e-1)^{2d_1}\tilde{\lambda}_{2d+1}, \ \tilde{\lambda}'^1=(e-1)^{2d_1}(e-2)^{2d_0}\tilde{\lambda}_{2d+1}&\text{if }\chi(e-2)\leq f-1;\end{array}$

\noindent(ii) 
$\begin{array}{l}\tilde{\lambda}^1=e^{2d_0+1}(e-1)^{2d_1}(e-2)^{2d_2}\cdots(e-j+1)^{2d_{j-1}}(e-j)^{2d_j-2}\tilde{\lambda}_{2d}\tilde{\lambda}_{2d+1},\\ \tilde{\lambda}'^1=(e-2)^{2d_0+1}(e-1)^{2d_1}(e-2)^{2d_2}\cdots(e-j+1)^{2d_{j-1}}(e-j)^{2d_j-2}\tilde{\lambda}_{2d}\tilde{\lambda}_{2d+1};\end{array}$

\noindent(iii) 
$\begin{array}{l}\tilde{\lambda}^1=(2f)e^{2d_0}(e-1)^{2d_1}(e-2)^{2d_2}\cdots(e-j+1)^{2d_{j-1}}(e-j)^{2d_j-2}\tilde{\lambda}_{2d}\tilde{\lambda}_{2d+1},\\ \tilde{\lambda}'^1=(2f-2)e^{2d_0}(e-1)^{2d_1}(e-2)^{2d_2}\cdots(e-j+1)^{2d_{j-1}}(e-j)^{2d_j-2}\tilde{\lambda}_{2d}\tilde{\lambda}_{2d+1}.\end{array}$

\subsection{}\label{pf-char2-2}
Assume that $p=2$ and $G=SO(V,Q)$ in the remainder of this section. Suppose that $\rc=(m,(\lambda,\chi))$. We keep the notations in \ref{pf-char2}.
 We show in this subsection that\\[5pt]
\indent\qquad (a) {\em $
f_N=\left\{\begin{array}{ll}1&\text{ if }m\geq f,\\
2&\text{ if }e-f<m<f,\\
m_{{\lambda}}(e)&\text{ if }m=0,\text{ or }0<m=e-f<f-1\text{ and }\rho=0,\\
m_{{\lambda}}(e)+1&\text{if }0<m=e-f=f-1\text{ and }\rho=0,\\
m_{{\lambda}}(e)+2&\text{if }0<m=e-f<f\text{ and }\rho\neq0.\end{array}\right.$}\\[5pt]
Recall that $f_N=\dim V_{\geq N} $ and $V_{\geq N}= V_{\geq -N+1}^\p\cap Q^{-1}(0)$ (see \ref{lem-o1}). We describe $V_{\geq N}$ in various cases in the following and then (a) follows. \\[5pt]
\indent  Suppose that $m=0$. Then $e\geq 2$ since $\xi\neq 0$. We have $V_{\geq N}=\{w+\sqrt{Q(w)}v_0|w\in\im T_\xi^{e-1}\}$, where $\sqrt{\quad}$ is a chosen square root on $\tk$. \\[5pt]
\indent Suppose that $m\geq f$. Then $V_{\geq N}=\spn\{v_0\}$. \\[5pt]
\indent  Suppose that $e-f<m<f$. Then $V_{\geq N}=\spn\{v_0\}\oplus L$, where $L=\{w\in W|Q(T_\xi^{f-1}w)=0\}^{\p}\cap W$ is a line (see \cite[5.5]{X2}). \\[5pt]
\indent Suppose that $0<m=e-f=f-1$. Then   $V_{\geq N}=\spn\{v_0\}\oplus L$, where $L=\{w\in W|T_\xi^{e-1}w=0,Q(T_\xi^{f-1}w)=0\}^{\p}\cap W$, and $\dim L=m_\lambda(e)$ if $\rho=0$, $\dim L=m_\lambda(e)+1$ if $\rho\neq0$ (see \cite[5.5]{X2}).\\[5pt]
\indent  Suppose that $0<m=e-f<f-1$ and $\rho\neq 0$. Then  $V_{\geq N}=\spn\{v_0\}\oplus L$, where $L=\{w\in W|T_\xi^{e-1}w=0,Q(T_\xi^{f-1}w)=0\}^{\p}\cap W$. Same argument as in \cite[5.5]{X2} shows that we have a decomposition $W=W_1\oplus W_2$ of $W$ into $T_\xi$-stable orthogonal subspaces such that $T_\xi^{e-1}W_2=0$ and $\chi_{T_\xi|_{W_2}}(e-1)=f$. Hence $L=\im T_\xi^{e-1}\oplus(\{x\in W_2|Q(T_\xi^{f-1}x)=0\}^{\p}\cap{W_2})$ and $\dim L=m_\lambda(e)+1$. \\[5pt]
\indent  Suppose that $0<m=e-f<f-1$ and $\rho=0$. Then  $V_{\geq N}=\{w+\beta(w,w_{**})v_0|w\in\im T_\xi^{e-1}\}$.

\subsection{}\label{pf-char2-5}We describe $\rc'=(m';(\lambda',\chi'))$  in various cases as follows.  Let  $j\geq 0$ be the unique integer such that  $\chi(e-j)=f\text{ and }\chi(e-j-1)<f.$\\[5pt]
\noindent(i) $m=0$, or $0<e-f=m<f-1$ and $\rho=0$. We have  \\[5pt]\indent{\em $m'=m$, $m_{\lambda'}(e)=0$, $m_{\lambda'}(e-2)=m_{\lambda}(e-2)+m_{\lambda}(e)$ (if $e>2$), $m_{\lambda'}(i)=m_{\lambda}(i)\text{ for }i\notin\{ e,e-2\}$, $\chi'(i)=\chi(i)$ for $i\leq e-1$.}

\noindent(ii) $m\geq f$. We have \\[5pt]\indent{\em $m'=m-1$, $m_{\lambda'}(i)=m_{\lambda}(i)\text{ for all }i$, $\chi'(\lambda_i)=\chi(\lambda_i)$ for $\lambda_i\leq e-1$, $\chi'(e)=f+1$  if $m=e-f$, and $\chi'(e)=f$  if $m>e-f$.}

\noindent (iii) $e-f<m<f$. We have\\[5pt]\indent{\em $m'=m-1,m_{\lambda'}(e-j)=m_\lambda(e-j)-2,\ m_{\lambda'}(e-j-1)=m_\lambda(e-j-1)+2$ (if $e>j+1$),$\ m_{\lambda'}(i)=m_{\lambda}(i)\text{ for }i\notin\{ e-j,e-j-1\}$, $\chi'(\lambda_i)=f-1\text{ for }e-j\leq\lambda_i\leq e-1,\ \chi'(\lambda_i)=\chi(\lambda_i)$ for $\lambda_i\leq e-j-1$, $\chi'(e)=f-1$ if $m\geq e-f+2$ and $\chi'(e)=f$ if $m=e-f+1$.}

 \noindent (iv) $0<e-f=m<f$ and $\rho\neq0$. We have\\[5pt]\indent{\em  $m'=m-1$, $m_{\lambda'}(e)=0,\ m_{\lambda'}(e-2)=m_{\lambda}(e-2)+m_{\lambda}(e)+2\delta_{j,1}-2\delta_{j,2}$ (if $e>2$), $ m_{\lambda'}(e-j)=m_{\lambda}(e-j)-2+\delta_{j,2}m_\lambda(e),\ m_{\lambda'}(e-j-1)=m_{\lambda}(e-j-1)+2+\delta_{j,1}m_\lambda(e)$ (if $e>j+1$), $ m_{\lambda'}(i)=m_{\lambda}(i)\text{ for }i\notin\{ e,e-2,e-j,e-j-1\}$, $\chi'(e-1)=f,\ \chi'(\lambda_i)=f-1\text{ for }e-j\leq\lambda_i\leq e-2,\ \chi'(\lambda_i)=\chi(\lambda_i)$ for $\lambda_i\leq e-j-1$.}

\noindent(v) $0<e-f=m=f-1$ and $\rho=0$. We have \\[5pt]\indent{\em $m'=m-1,m_{\lambda'}(e)=0,\ m_{\lambda'}(e-2)=m_\lambda(e)+m_\lambda({e-2})\ (\text{if }e>2),\ \ m_{\lambda'}(i)=m_{\lambda}(i)\text{ for }i\notin\{e,e-2\}$, $\chi'(e-1)=f,\  \chi'(e-2)=f-1,\  \chi'(\lambda_i)=\chi(\lambda_i)$ for $\lambda_i\leq e-2$.}

 Recall that we can choose $u_0$ (or $W$ if $m=0$) such that $\chi_W(\lambda_i)=\chi(\lambda_i)$ in the decomposition $V=\spn\{u_i,v_i\}\oplus W$. In the following $u_0$ or $W$ is chosen as such. Let $m_\lambda(e-i)=2d_i$.

Assume that we are in case (i). Suppose first that $m=0$. Then $e=f$.  There exists $w_0\in W$ such that $\beta(w_0,w)^2=Q(w)$ for all $w\in W$. Let $\tilde{W}=\{w+\beta(w_0,w)v_0|w\in W\}$. Then $V_{\geq -N+1}=\spn\{v_0\}\oplus\ker \tilde{T}_\xi^{e-1}$ and $V_{\geq N}=\im \tilde{T}_\xi^{e-1}$ (note that $Q(\tilde{T}_\xi^{e-1}\tilde{W})=0$). The description for $\rc'$ follows.
Suppose now that $0<e-f=m<f-1$ and $\rho=0$. Let $\tilde{u}_0=u_0+w_{**}$ and let $\tilde{u}_i,\tilde{W},\tilde{T}_\xi$ be defined accordingly. Then for all $\tilde{w}\in\tilde{W}$, $Q(\tilde{T}_\xi^{f-1}\tilde{w})=\beta(T_\xi^{e-1}w_{**},\pi_W(\tilde{w}))^2+Q(T_\xi^{f-1}\pi_W(\tilde{w}))=0$. Then  $V'=\spn\{v_i,i\in[0,m],\tilde{u}_i,i\in[0,m-1]\}\oplus \tilde{W}'$, where $\tilde{W}'=(\ker{\tilde{T}}_\xi^{e-1}/\im {\tilde{T}}_\xi^{e-1})$. Thus $m'=m$. Let $T_{\xi'}:\tilde{W}'\to \tilde{W}'$ be defined for $\alpha_{\xi'}$. We have ${\chi_{\tilde{W}}}(i)\leq f-1$. Suppose that $\chi_{\tilde{W}}(e)=\tilde{f}$.
We have a decomposition $\tilde{W}=W_1\oplus W_2$ of $\tilde{W}$ into mutually orthogonal $\tilde{T}_\xi$-stable subspaces such that (see \cite[5.6]{X3})\\[5pt]
\indent\hspace{1.25in} $\tilde{T}_\xi|_{W_1}=W_{\tilde{f}}(e)^{{d_0}}$ and  $T_\xi^{e-1}W_2=0$.\\[5pt] Then $\tilde{W}'=W_1'\oplus W_2$, where $W_1'=(\ker \tilde{T}_\xi^{e-1}\cap W_1)/ \tilde{T}_\xi^{e-1}W_1$ and $T_{\xi'}|_{W_1'}=W_{\tilde{f}-1}(e-2)^{{d_0}}$.   We have $\chi'(i)=\max(i-m,\chi_{T_{\xi'}|_{W_1'}}(i),\chi_{\tilde{T}_\xi|_{W_2}}(i))$ and $\chi(i)=\max(i-m,\chi_{\tilde{T}_\xi|_{W_1}}(i),\chi_{\tilde{T}_\xi|_{W_2}}(i))$. Thus $\chi'(e-1)=f-1=\chi(e-1)$ and $\chi'(i)=\max(i-m,\chi_{\tilde{T}_\xi|_{W_2}}(i))=\chi(i)$ (note that $\chi_{T_{\xi'}|_{W_1'}}(i)\leq\max(0,i-m)$ and  $\chi_{T_\xi|_{W_1}}(i)\leq\max(0,i-m)$) for all $i\leq e-2$.

Assume that we are in case (ii). Then $V'=(\spn\{v_i,u_i,i\in[0,m-1],v_m\}/\spn\{v_0\})\oplus W$. Thus $m'=m-1$ and $m_{\lambda'}(i)=m_\lambda(i)$. We have $\chi'(\lambda_i)=\max(\lambda_i-m+1,\chi(\lambda_i))$. If $m>\lambda_i-\chi(\lambda_i)$, then  $\chi'(\lambda_i)=\chi(\lambda_i)$; if $m=\lambda_i-\chi(\lambda_i)$, then $m=e-f$ and thus $e=2m$, $\lambda_i=e$ (note that $\chi(\lambda_i)\geq \lambda_i/2>\lambda_i-m$ for $\lambda_i<e=2m$) and $\chi'(e)=f+1$.

In cases (iii)-(v), we have  $V'=(\spn\{v_i,u_i,i\in[0,m-1],v_m\}/\spn\{v_0\})\oplus W'$, where $W'=\Lambda_W/(\Lambda_W^{\p}\cap W)$ for some subspace $\Lambda_W\subset W$. Thus $m'=m-1$. Let $T_{\xi'}:W'\to W'$ be defined for $\alpha_{\xi'}$. We write $T_{\xi'}=(\lambda',\chi_{W'}')$. We can apply the results in \cite[5.6]{X3} for $T_\xi$ and $T_{\xi'}$ and then in each case the assertions on $m_{\lambda'}(i)$ follow. The assertions on $\chi'$ also hold since $\chi'(\lambda_i')=\max(\lambda_i'-m+1,\chi_{W'}'(i))$ (see below for the description of $\chi_{W'}'(i)$).\\[5pt]
\indent(iii) In this case  $\Lambda_W=\{w\in W|Q(T_\xi^{f-1}w)=0\}$ and $\chi_{W'}'(e-k)=f-1\text{ for }k\in[0,j],\ \chi_{W'}'(\lambda_i)=\chi(\lambda_i)\text{ for }\lambda_i\leq e-j-1$. Note that $\lambda_i-m+1\leq\chi(\lambda_i)$ for $\lambda_i\leq e- j-1$. \\[5pt]
\indent(iv) In this case   $\Lambda_W=\{w\in W|T_\xi^{e-1}w=0,Q(T_\xi^{f-1}w)=0\}$, and $\chi'_{W'}(e-k)=f-1\text{ for }k\in[0,j],\ \chi_{W'}'(\lambda_i)=\chi(\lambda_i)\text{ for }\lambda_i\leq e-j-1$. Since $\rho\neq 0$,  $\chi(e-1)=f$ and thus  $m>\lambda_i-\chi(\lambda_i)$ for all $\lambda_i\leq e-1$. \\[5pt]
\indent(v) In this case  $\Lambda_W=\{w\in W|T_\xi^{e-1}w=0\}$ and $\chi_{W'}'(\lambda_i)=\chi(\lambda_i)\text{ for }\lambda_i\leq e-1.$ Note that for $\lambda_i\leq e-2$, $\chi(\lambda_i)\geq [\frac{\lambda_i+1}{2}]\geq \lambda_i-m+1$.

\subsection{}\label{pf-char2-6}Using the definition of $\Psi_{R}^*$ and the  description of $\rc'$ in \ref{pf-char2-5}, we compute $\Psi_{R}^*(\rc)=\tilde{\lambda}$ and $\Psi_{R'}^*(\rc')=\tilde{\lambda}'$ in each case (i)-(v) as follows. It is then easy to check that \ref{pf-char2} (b) holds in each case.

Let $d=d_0+d_1$ in cases (i) and (v); $d=d_0$ in case (ii); and $d=\sum_{a\in[0,j]}d_a$ in cases (iii) and (iv). Then $\tilde{\lambda}'_i=\tilde{\lambda}_i$ for all $i\geq 2d+3$ since $\mu_i=\mu_i'$ and $\nu_{i-1}=\nu_{i-1}'$ for all $i\geq d+2$.
In case (i), if $\chi(e-2)=f-2$, then $\mu_{d+1}=\mu_{d+1}'$ and thus $\tilde{\lambda}_{2d+2}'=\tilde{\lambda}_{2d+2}'$; if moreover $f\geq m+3$, then $\tilde{\lambda}_{2d+1}'=\tilde{\lambda}_{2d+1}'$ since $\mu_{d+1}<\nu_{d}$, $\mu_{d+1}'<\nu_{d}'$. In case (v) or in case (ii) with $m=e-f$ (then $e=2f$), $\tilde{\lambda}_{2d+2}=\tilde{\lambda}_{2d+2}'$ since $\nu_{d+1}<\mu_{d+1}+2$ and $\nu_{d+1}'<\mu_{d+1}'+2$; if $m>e-f$, then $\tilde{\lambda}_i=\tilde{\lambda}_i'$ for $i=2d+1,2d+2$ since $\mu_{d+1}=\mu_{d+1}'$ and $\nu_{d}=\nu_{d}'$. In case (iii) and (iv), $\tilde{\lambda}_i=\tilde{\lambda}_i'$ for $i=2d+1,2d+2$ since $\mu_{d+1}=\mu_{d+1}'$ and $\mu_{d+1}<\nu_{d}$, $\mu_{d+1}'<\nu_{d}'$.

Let $\tilde{\lambda}^1=(\tilde{\lambda}_1,\ldots,\tilde{\lambda}_{2d+2})$ and $\tilde{\lambda}'^1=(\tilde{\lambda}'_1,\ldots,\tilde{\lambda}'_{2d+2})$. We have

\noindent(i)  
$\begin{array}{ll}\tilde{\lambda}^1=e^{2d_0}(e-1)^{2d_1}\tilde{\lambda}_{2d+1}\tilde{\lambda}_{2d+2},\\ \tilde{\lambda}'^1=(e-1)^{2d_1}(e-2)^{2d_0}\tilde{\lambda}_{2d+1}\tilde{\lambda}_{2d+2}&\text{if }\chi(e-2)=f-2\geq m+1\\
\tilde{\lambda}^1=e^{2d_0}(e-1)^{2d_1+1}\tilde{\lambda}_{2d+2}, \ \tilde{\lambda}'^1=(e-1)^{2d_1+1}(e-2)^{2d_0}\tilde{\lambda}_{2d+2}&\text{if }\chi(e-2)=f-2=m\\
\tilde{\lambda}^1=e^{2d_0}(e-1)^{2d_1+2}, \ \tilde{\lambda}'^1=(e-1)^{2d_1+2}(e-2)^{2d_0}&\text{if }\chi(e-2)=f-1.\end{array}$

\noindent(ii) 
$\begin{array}{ll}\tilde{\lambda}^1=(e+1)e^{2d_0}\tilde{\lambda}_{2d+2},\ \tilde{\lambda}'^1=e^{2d_0}(e-1)\tilde{\lambda}_{2d+2} &\text{if }m=e-f\\
\tilde{\lambda}^1=(2m+1)(2f-1)e^{2d_0-2}\tilde{\lambda}_{2d+1}\tilde{\lambda}_{2d+2},\\ \tilde{\lambda}'^1=(2m-1)(2f-1)e^{2d_0-2}\tilde{\lambda}_{2d+1}\tilde{\lambda}_{2d+2}&\text{if }m>e-f\text{ and }e<2f;\\
\tilde{\lambda}^1=(2m+1)e^{2d_0-1}\tilde{\lambda}_{2d+1}\tilde{\lambda}_{2d+2},\ \tilde{\lambda}'^1=(2m-1)e^{2d_0-1}\tilde{\lambda}_{2d+1}\tilde{\lambda}_{2d+2}&\text{if }m>e-f\text{ and }e=2f.\end{array}$

\noindent(iii)
$\begin{array}{ll}\tilde{\lambda}^1=(m+f)^2e^{2d_0}\cdots(e-j)^{2d_j-2}\tilde{\lambda}_{2d+1}\tilde{\lambda}_{2d+2}, \\ \tilde{\lambda}'^1=(m+f-2)^2e^{2d_0}\cdots(e-j)^{2d_j-2}\tilde{\lambda}_{2d+1}\tilde{\lambda}_{2d+2}\end{array}.$

\noindent (iv) 
$\begin{array}{l}\tilde{\lambda}^1=e^{2d_0+2}(e-1)^{2d_1}(e-2)^{2d_2}\cdots(e-j+1)^{2d_{j-1}}(e-j)^{2d_j-2},\\ \tilde{\lambda}'^1=(e-2)^{2d_0+2}(e-1)^{2d_1}(e-2)^{2d_2}\cdots(e-j+1)^{2d_{j-1}}(e-j)^{2d_j-2}.\end{array}$

\noindent(v) 
$\tilde{\lambda}^1=e^{2d_0+1}(e-1)^{2d_1}\tilde{\lambda}_{2d+2},\ \tilde{\lambda}'^1=(e-1)^{2d_1}(e-2)^{2d_0+1}\tilde{\lambda}_{2d+2}.$

\section{Nilpotent coadjoint orbits in type $G_2$ in characteristic 3 and in type $F_4$ in characteristic 2}\label{sec-excep}
Assume that $G$ is of type $G_2$ and $p=3$, or $G$ of type $F_4$ and $p=2$  in this section unless otherwise stated. 
We classify the nilpotent coadjoint orbits in $\Lg^*$ and determine the closure relation among them.  We describe explicitly
the nilpotent pieces in $\Lg^*$ defined in \cite{CP}. In particular, it follows from the classification (see Subsections \ref{ssec-ex1}-\ref{ssec-exn}) that
\begin{proposition}\label{prop-number}
 The number of nilpotent coadjoint orbits in $\Lg^*$ is finite.
\end{proposition}
As mentioned in the Introduction, Proposition is true now for any connected reductive algebraic group $G$ in any characteristic.

\subsection{}Following a suggestion of the referee, we include here references for unipotent orbits in $G$ and nilpotent orbits in $\Lg$. When $p\neq 2,3$, we can identify nilpotent orbits in $\Lg$ with unipotent orbits in $G$, which are thus both classified by Bala-Carter theory. If $G$ is of type $G_2$ and $p=2,3$, both unipotent and nilpotent orbits are classified by Stuhler in \cite{S}. If $G$ is of type $F_4$, the unipotent orbits  are classified by Shoji  in \cite{Sh} when $p=3$ and by Shinoda in \cite{Shi} when $p=2$, and the nilpotent orbits in $\Lg$  are classified by Spaltenstein in \cite{Sp4} when $p=2$ and by Holt and Spaltenstein in \cite{HS} when $p=3$. The closure relation among unipotent orbits in $G$ is determined by Spaltenstein in \cite{S5} and that among nilpotent orbits in $\Lg$ when $G$ is of type $F_4$ and $p=2$ is given in \cite{Sp4}.

\subsection{}\label{ssec-ex1}
Let $\tF_q$ be a finite field of characteristic $3$ (resp. $2$). Let $G$ be of type $G_2$ (resp. $F_4$) defined over $\tF_q$. We prove Proposition \ref{prop-number} by studying $G(\tF_q)$-orbits in $\cN_{\Lg^*}(\tF_q)$. The strategy is as follows. We specify various elements $\xi\in\cN_{\Lg^*}(\tF_q)$ which lie in different $G(\tF_q)$-orbits and compute $|Z_G(\xi)(\tF_q)|$ (the number of rational points in the centralizer $Z_G(\xi)$).
 Then by a direct computation one verifies that the numbers of rational points in all nilpotent coadjoint orbits add up to $q^{2N}$, where $N$ is the number of positive roots. As $|\cN_{\Lg^*}(\tF_{\mathbf{q}})|=q^{2N}$ (private communication by G. Lusztig, see also \cite{CP}),  we get all the $G(\tF_q)$-orbits in $\cN_{\Lg^*}(\tF_q)$.

\subsection{} \label{ssec-cen}We describe how to compute $|Z_G(\xi)(\tF_q)|$ for $\xi\in\cN_{\Lg^*}(\tF_q)$ in this subsection. 

Let $T$ be a maximal torus of $G$, $R$ the root system of $(G,T)$, $\Pi\subset R$ a set of simple roots, and $R^+\subset R$ the corresponding set of positive roots.
 We have a Chevalley basis $\{h_\alpha,\ \alpha\in \Pi;\ e_\alpha,\ \alpha\in R\}$ of $\Lg$ satisfying $$[h_\alpha,h_\beta]=0;\ [h_\alpha,e_\beta]=A_{\alpha \beta}e_\beta;\ [e_\alpha,e_{-\alpha}]=h_\alpha;\ [e_\alpha,e_\beta]=N_{\alpha,\beta}e_{\alpha+\beta},\text{ if }\alpha+\beta\in R,$$
where $A_{\alpha,\beta}$ and $N_{\alpha,\beta}$ are constant integers  (for determination of structural constants $N_{\alpha,\beta}$ see \cite{Ch}).
For each $\alpha\in R$, there is a unique 1-dimensional connected closed unipotent subgroup $U_\alpha\subset G$ and an isomorphism 
$$x_\alpha:\ \mathbb{G}_a\to U_\alpha$$ such that $sx_\alpha(t)s^{-1}=x_\alpha(\alpha(s)t)$ for all $s\in T$ and $t\in\mathbb{G}_a$. We assume that $dx_\alpha(1)=e_\alpha$ and $n_\alpha(t):=x_\alpha(t)x_{-\alpha}(-t^{-1})x_\alpha(t)$ normalizes $T$. Define $h_\alpha(t)=n_\alpha(t)n_\alpha(-1)$. Then $T$ is generated by $h_\alpha(\lambda),\alpha\in \Pi,\lambda\in\tk^\times$. Let $B$ be the Borel subgroup $UT$ of $G$, where $\displaystyle{U=\{\prod_{\alpha\in R^+}x_\alpha(t_\alpha)\,|\,t_\alpha\in\mathbb{G}_a\}}$.
By Bruhat decomposition, each $g\in G$ can be written uniquely in the form $g=bn_wu_w$ for some $w\in W=N_G(T)/T$, some $b\in B$ and some $\displaystyle{u_w\in U_w:=\{\prod_{\alpha>0,w(\alpha)<0}x_\alpha(t_\alpha)\,|\,t_\alpha\in\mathbb{G}_a\}}$, where $n_w$ is a representative of $w$ in $N_G(T)$. We can
choose $n_\alpha=n_\alpha(1)$ to be the representative of the simple reflection $s_\alpha\in W$, $\alpha\in \Pi$.
Let $\Lt$, $\Lb$ be the Lie algebra of $T$, $B$ respectively. We define $e_\alpha'\in\Lg^*$ by $$e_\alpha'(\Lt)=0;\ e_\alpha'(e_\beta)=\delta_{-\alpha,\beta},\forall\ \beta\in R.$$
Then $\{e_\alpha',\alpha\in R^+\}$ form a basis of $\Ln^*=\{\xi\in\Lg^*\,|\,\xi(\Lb)=0\}$. The coadjoint actions of $x_\alpha(t)$, $\alpha\in R^+$, $h_\alpha(\lambda)$ and $n_\alpha$, $\alpha\in\Pi$ on $e_\beta'$, $\alpha,\beta\in R^+$ are given as follows
\begin{eqnarray*}
\text{(a)}&&x_\alpha(t).e_\beta'=\sum_{i}(-1)^it^iM_{\alpha,-i\alpha-\beta,i}e_{i\alpha+\beta}',\ \beta\neq \alpha,\ x_\alpha(t).e_\alpha'=e_\alpha',\\
&&h_\alpha(\lambda).e_\beta'=\lambda^{A_{\alpha\beta}}e_\beta',\ n_\alpha.e_\beta'=\pm e_{s_\alpha(\beta)}',\ \alpha\in\Pi,
\end{eqnarray*}
where
$\displaystyle{
M_{\alpha,\beta,i}=\frac{1}{i!}N_{\alpha,\beta}N_{\alpha,\alpha+\beta}\cdots N_{\alpha,(i-1)\alpha+\beta}}
$ (here the equality is in $\mathbb{N}$ and we then reduce mod $p$ to regard $M_{\alpha,\beta,i}$ as a constant in $\tk$).

Since $\cN_{\Lg^*}=G.\Ln^*$, we can find representatives of nilpotent coadjoint orbits in $\Ln^*$, namely, we can choose elements of the form  $\xi=\sum_{\alpha\in R^+}a_\alpha e_\alpha'$, $a_\alpha\in\tF_q$. Now we can compute $|Z_G(\xi)(\tF_q)|$  using the Bruhat decomposition 
and (a). In particular, we need knowledge on the set $\{w\in W\,|\, Z_G(\xi)\cap(BwB)\neq\emptyset\}$.  Let\\[2pt]
\centerline {$\Delta^\xi_{\text{min}}$ be the set of minimal elements in the set $\{\alpha\in R^+\,|\,a_\alpha\neq 0\}$}\\[2pt]
under the order relation $>$ on $R^+$, where $\alpha>\beta$ if $\alpha-\beta$ can be written as a  sum of positive roots.
If $Z_G(\xi)\cap BwB\neq \emptyset$, then by   (a), for any $\alpha\in \Delta_{\min}^\xi$, there exists $\beta\in\Delta_{\min}^\xi$ such that $w(\alpha)\geq \beta$, more precisely,  
$$w(\alpha)\in\Delta^\xi=\{\beta\in R^+\,|\,c_\beta(b)\neq 0\text{ for some }b\in B\},$$ where we write $b.\xi=\sum_{\beta\in R^+}c_\beta(b) e_\beta'$ for $b\in B$. 
\subsection{}\label{ssec-g2}
Suppose that $G$ is of type $G_2$ and $p=3$ in this subsection. 
We denote by
$\alpha$ (resp. $\beta$) the short (resp. long) simple root. The structural constants can be chosen as follows
\begin{eqnarray*}
&&N_{\alpha,\beta}=1,\ N_{\alpha,\alpha+\beta}=2,\ N_{\alpha,2\alpha+\beta}=3,\ N_{\beta,3\alpha+\beta}=-1,\
N_{\alpha+\beta,2\alpha+\beta}=3.
\end{eqnarray*}
Fix $\zeta\in\tF_{\mathbf{q}}\backslash\{x^2\,|\,x\in\tF_{\mathbf{q}}\}$ and $\varpi\in\tF_{\mathbf{q}}\backslash\{x^3+x\,|\,x\in\tF_{\mathbf{q}}\}$. 
The representatives $\xi$ for nilpotent coadjoint orbits over $\tF_q$ and $|Z_G(\xi)(\tF_{\mathbf{q}})|$ are listed in Table \ref{tab:1}. 

\begin{table}[H]
\centering
\begin{tabular}{c l l}
\hline
Orbit&Representative $\xi$& $|Z_G(\xi)(\tF_{\mathbf{q}})|$\\
\hline
$G_2$&$\xi_1=e_\alpha'+e_\beta'$&$\mathbf{q}^2$\\
$G_2(a_1)$&$\xi_2=e_\beta'+e_{2\alpha+\beta}'$&$6\mathbf{q}^4$\\
$G_2(a_1)$&$\xi_{2,2}=e_\beta'+e_{2\alpha+\beta}'-\varpi e_{3\alpha+\beta}'$&$3\mathbf{q}^4$\\
$G_2(a_1)$&$\xi_{2,3}=e_\beta'-\zeta e_{2\alpha+\beta}'$&$2\mathbf{q}^4$\\
$\widetilde{A_1}$&$\xi_3=e_\alpha'$&$\mathbf{q}^4(\mathbf{q}^2-1)$\\
${A_1}$&$\xi_4=e_\beta'$&$\mathbf{q}^6(\mathbf{q}^2-1)$\\
$\emptyset$ &$\xi_5=0$&$\mathbf{q}^6(\mathbf{q}^2-1)(\mathbf{q}^6-1)$\\
\hline\\
\end{tabular}
\caption{Nilpotent coadjoint orbits in $\Lg^*(\tF_q)$, type $G_2$, $p=3$.}
\label{tab:1}
\end{table}
One can easily verify that $\xi_2$, $\xi_{2,2}$,  and $\xi_{2,3}$ are in the same $G$-orbit. Thus $\xi_1,\xi_2,\xi_3,\xi_4,\xi_5$ form  a set of representatives for $G$-obits in $\cN_{\Lg^*}$. This proves Proposition \ref{prop-number} for type $G_2$.

It is easy to verify that the closure relation among nilpotent coadjoint orbits in $\Lg^*$ is as in Figure \ref{fig1.tag} and  the nilpotent pieces in $\Lg^*$ coincide with nilpotent coadjoint orbits.

\begin{figure}[H]
\centering
\includegraphics[height=50mm]{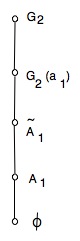}
\caption{Closure relation among nilpotent coadjoint orbits in $\Lg^*$, type $G_2$, $p=3$.}

\label{fig1.tag}
\end{figure}

\subsection{}\label{ssec-exn}Suppose that $G$ is of type $F_4$ and $p=2$ in this subsection. We denote by $p,q$ (resp. $r,s$) the  long (resp. short) simple roots  with $(q,r)\neq 0$.
We denote by $apbqcrds$ the root $ap+bq+cr+ds$.  The structural constants can be chosen as follows:
\begin{eqnarray*}
&&N_{p,q}=N_{p,qr}=N_{p,q2r}=N_{p,p3q4r2s}=N_{p,qrs}=N_{p,q2rs}=N_{p,q2r2s}=N_{q,rs}=N_{q,pq2r}=N_{q,pq2rs}\\
&&\quad=N_{q,pq2r2s}=N_{q,r}=N_{q,p2q4r2s}=N_{r,s}=N_{r,qrs}=N_{r,pqrs}=N_{r,p2q2rs}=N_{pq,rs}=N_{pq,p2q4r2s}\\
&&\quad=N_{s,q2r}=N_{s,pq2r}=N_{s,p2q2r}
=N_{rs,pqr}=N_{rs,p2q2r}=N_{qr,pqrs}=N_{qr,pq2r2s}=N_{q2r,pq2r2s}\\&&\quad=N_{q2r,p2q2r2s}=N_{pq2r,p2q2r2s}=N_{pqr,q2rs}=N_{qrs,pq2rs}=1,\\&&N_{r,pq}=N_{r,p2q2r2s}=N_{s,qr}=N_{s,pqr}=N_{s,p2q3rs}=N_{pq,q2r}=N_{pq,q2rs}=N_{pq,q2r2s}=N_{qr,rs}\\&& 
\quad=N_{qr,pq2rs}=N_{rs,p2q2rs}=N_{q2r,pqrs}=N_{pqr,q2r2s}=N_{pqr,qrs}=N_{pq2r,q2r2s}=N_{qrs,pq2r}\\
&&\quad=N_{pqrs,q2rs}=N_{p2q2r,pq2r2s}=N_{p2q2r,q2r2s}=-1,\\
&&N_{s,pq2rs}=N_{s,p2q2rs}=N_{s,q2rs}=N_{rs,p2q3rs}=N_{qrs,p2q3rs}=N_{pqrs,p2q3rs}=2,
\end{eqnarray*}
\begin{eqnarray*}
&&N_{r,pqr}=N_{r,qr}=N_{r,p2q3r2s}=N_{qr,p2q3r2s}=N_{qr,pqr}=
N_{rs,pqrs}=N_{rs,qrs}\\&&\quad=N_{pqr,p2q3r2s}=N_{qrs,pqrs}=N_{q2rs,pq2rs}=N_{q2rs,p2q2rs}=N_{pq2rs,p2q2rs}=-2.\end{eqnarray*}
Fix $\eta\in\tF_{\mathbf{q}}\backslash\{x^2+x\,|\,x\in \tF_{\mathbf{q}}\}$ and $\varpi\in\tF_{\mathbf{q}}\backslash\{x^3+x\,|\,x\in \tF_{\mathbf{q}}\}$. The representatives $\xi$ for nilpotent coadjoint orbits over $\tF_q$ and $|Z_G(\xi)(\tF_{\mathbf{q}})|$ are listed in Table \ref{tab:3}. 

\begin{table}[H]
\centering
\begin{tabular}{c l l}
\hline
Orbit&Representative $\xi$& $|Z_G(\xi)(\tF_{\mathbf{q}})|$\\
\hline
$F_4$&$\xi_1=e_p'+e_q'+e_r'+e_s'$&$\mathbf{q}^4$\\
$F_4(a_1)$&$\xi_2=e_p'+e_{qr}'+e_{q2r}'+e_s'$&$2\mathbf{q}^6$\\
$F_4(a_1)$&$\xi_{2,2}=e_p'+e_q'+e_{qr}'+e_s'+\eta e_{q2r}'$&$2\mathbf{q}^6$\\
$F_4(a_2)$&$\xi_{3}=e_{pq}'+e_{qr}'+e_{rs}'+e_{q2r}'$&$\mathbf{q}^8$\\
$B_3$&$\xi_{4}=e_p'+e_{qrs}'+e_{q2r}'+e_{pq2rs}'$&$\mathbf{q}^{10}$\\
$C_3$&$\xi_5=e_s'+e_{q2r}'+e_{pqr}'$&$\mathbf{q}^8(\mathbf{q}^2-1)$\\
$F_4(a_3)$&$\xi_6=e_{pqr}'+e_{qrs}'+e_{pq2r}'+e_{q2r2s}'$&$24\mathbf{q}^{12}$\\
$F_4(a_3)$&$\xi_{6,2}=e_{pq}'+e_{pqr}'+e_{q2rs}'+e_{q2r2s}'+\eta e_{pq2r}'$&$8\mathbf{q}^{12}$\\
$F_4(a_3)$&$\xi_{6,3}=e_{pqr}'+e_{qrs}'+e_{pq2r}'+e_{q2r2s}'+\eta e_{pq2r2s}'$&$4\mathbf{q}^{12}$\\
$F_4(a_3)$&$\xi_{6,4}=e_{pq}'+e_{pqr}'+e_{q2rs}'+e_{q2r2s}'+\eta e_{q}'$&$4\mathbf{q}^{12}$\\
$F_4(a_3)$&$\xi_{6,5}=e_{pqr}'+e_{qrs}'+e_{q2r}'+e_{q2r2s}'+\varpi e_{pq2r2s}'$&$3\mathbf{q}^{12}$\\
$(B_3)_2$&$\xi_7=e_p'+e_{qr}'+e_{q2r2s}'$&$\mathbf{q}^{10}(\mathbf{q}^2-1)$\\
$C_3(a_1)$&$\xi_8=e_{pqr}'+e_{q2rs}'+e_{q2r2s}'$&$2\mathbf{q}^{12}(\mathbf{q}^2-1)$\\
$C_3(a_1)$&$\xi_{8,2}=e_{pq}'+e_{pqr}'+e_{q2rs}'+\eta e_{pq2r}'$&$2\mathbf{q}^{12}(\mathbf{q}^2-1)$\\
$B_2$&$\xi_9=e_{pqr}'+e_{q2r2s}'$&$2\mathbf{q}^{12}(\mathbf{q}^2-1)^2$\\
$B_2$&$\xi_{9,2}=e_{pq}'+e_{pqr}'+e_{q2r2s}'+\eta e_{pq2r}'$&$2\mathbf{q}^{12}(\mathbf{q}^4-1)$\\
$\widetilde{A_2}+A_1$&$\xi_{10}=e_{pqrs}'+e_{q2rs}'+e_{p2q2r}'$&$\mathbf{q}^{14}(\mathbf{q}^2-1)$\\
$A_2+\widetilde{A_1}$&$\xi_{11}=e_{p2q2r}'+e_{q2r2s}'+e_{pq2rs}'$&$\mathbf{q}^{16}(\mathbf{q}^2-1)$\\
$\widetilde{A_2}$&$\xi_{12}=e_{pqrs}'+e_{q2rs}'$&$\mathbf{q}^{14}(\mathbf{q}^2-1)(\mathbf{q}^6-1)$\\
$A_2$&$\xi_{13}=e_{p2q2r}'+e_{pq2r2s}'+e_{p2q3r2s}'$&$\mathbf{q}^{20}(\mathbf{q}^2-1)$\\
$A_1+\widetilde{A_1}$&$\xi_{14}=e_{p2q2r2s}'+e_{p2q3rs}'$&$\mathbf{q}^{20}(\mathbf{q}^2-1)^2$\\
$(A_2)_2$&$\xi_{15}=e_{p2q2r}'+e_{pq2r2s}'$&$\mathbf{q}^{20}(\mathbf{q}^2-1)(\mathbf{q}^6-1)$\\
$\widetilde{A_1}$&$\xi_{16}=e_{p2q3r2s}'$&$2\mathbf{q}^{21}(\mathbf{q}^2-1)(\mathbf{q}^3-1)(\mathbf{q}^4-1)$\\
$\widetilde{A_1}$&$\xi_{16,2}=e_{p2q2r2s}'+e_{p2q3r2s}'+\eta e_{p2q4r2s}'$&$2\mathbf{q}^{21}(\mathbf{q}^2-1)(\mathbf{q}^3+1)(\mathbf{q}^4-1)$\\
$A_1$&$\xi_{17}=e_{2p3q4r2s}'$&$\mathbf{q}^{24}(\mathbf{q}^2-1)(\mathbf{q}^4-1)(\mathbf{q}^6-1)$\\
$\emptyset$&$\xi_{18}=0$&$\mathbf{q}^{24}(\mathbf{q}^2-1)(\mathbf{q}^6-1)(\mathbf{q}^8-1)(\mathbf{q}^{12}-1)$\\
\hline\\
\end{tabular}
\caption{Nilpotent coadjoint orbits in $\Lg^*(\tF_q)$, type $F_4$, $p=2$.}
\label{tab:3}
\end{table}

The computations of $|Z_G(\xi)(\tF_q)|$ are long and follow the strategy described in \ref{ssec-cen}. We give one example here and omit the details. We denote by $n_p,n_q,n_r,n_s$ the simple reflections in $W$. Consider $\xi_{7}=e_p'+e_{qr}'+e_{q2r2s}'$. We have $\Delta^{\xi_7}=R^+\backslash\{q,r,s,rs\}$ and $\Delta^{\xi_7}_{\min}=\{p,qr\}$ (see \ref{ssec-cen}). Assume $Z_G(\xi_7)\cap BwB\neq \emptyset$. Then $w(p)\in\Delta^{\xi_7}$ and $w(qr)\in\Delta^{\xi_7}$. In fact a further look at the formulas in \ref{ssec-cen} (a) shows that if either $w(qrs)\notin \Delta^{\xi_7}$, or both $w(q2r)\notin \Delta^{\xi_7}$ and $w(q2rs)\notin \Delta^{\xi_7}$, then $w(q2r2s)\in \Delta^{\xi_7}$. This forces $w\in\la n_r,n_s\ra$. Now it is easy to verify that
$$|(Z_G(\xi_{7})\cap BwB)(\tF_{\mathbf{q}})|=\left\{\begin{array}{ll}\mathbf{q}^9(\mathbf{q}-1)&\text{ if }w=1,n_r\\
\mathbf{q}^9(\mathbf{q}-1)^2&\text{ if }w=n_s,n_sn_r,n_rn_s\\\mathbf{q}^9(\mathbf{q}-1)^3&\text{ if }w=n_rn_sn_r.\end{array}\right.$$
Thus $|Z_G(\xi_7)(\tF_q)|=q^{10}(q^2-1).$

We need to show that $\xi,\xi'$ are not in the same $G(\tF_q)$-orbit for those $\xi,\xi'$ in Table 2 with $|Z_G(\xi)(\tF_q)|=|Z_G(\xi')(\tF_q)|$.
We verify this for $\xi_2$ and $\xi_{2,2}$. The verification of others is entirely similar. Assume that there exists $g\in G(\tF_q)$ such that $g.\xi_2=\xi_{2,2}$ and $g\in BwB$. By a similar argument as in the last paragraph of Subsection \ref{ssec-cen}, we have $w(p)>0,w(qr)>0,w(s)>0$ and $w^{-1}(p)>0,w^{-1}(q)>0,w^{-1}(s)>0$. A closer look at the formulas in \ref{ssec-cen} (a) shows that $w(rs)$ and $w(q2r)$ can not both be negative if $w(r)<0$. This forces $w=1$ or $w=n_r$. It is clear that $g\notin B$. Thus $g=bn_{r}x_r(\sigma)$ for some $b\in B$ and $\sigma\in\tF_q$. When using (a) to solve $g.\xi_2=\xi_{2,2}$, the following equation appears
$$\sigma^2+\sigma+\eta=0,$$
which has no solution in $\tF_q$ by our choice of $\eta$. Moreover the computation shows that $\xi_2$ and $\xi_{2,2}$ lie in the same $G$-orbit.

Similarly one can verify that $\{\xi_{6,i},i=1,\ldots,5\}$, $\{\xi_8,\xi_{8,2}\}$, $\{\xi_9,\xi_{9,2}\}$,
$\{\xi_{16},\xi_{16,2}\} $ are in the same $G$-orbit respectively. Hence $\xi_i$, $i\in\{1,\ldots,18\}$ form a set of  representatives for $G$-obits in $\cN_{\Lg^*}$.  This completes the classification of nilpotent coadjoint orbits in $\Lg^*$ for type $G_2$, $F_4$ and the proof of Proposition \ref{prop-number}.

\subsection{}Suppose again that $G$ is of type $F_4$ and $p=2$ in this subsection. 
For $\Sigma\subset R^+$, we define $$S(\Sigma)=\text{span}\{e_\alpha',\alpha\in\Sigma\}\subset{\Lg^*}.$$
Now let $S_i=S(\Sigma_i)$, $1\leq i\leq 15$ and $S_{16}=\{0\}$, where
\begin{eqnarray*}
&&\Sigma_1=R^+,\ \Sigma_2=R^+\backslash\{r\},\ \Sigma_3=R^+\backslash\{p,r\},\ \Sigma_4=R^+\backslash\{r,s,rs\},\ \Sigma_5=R^+\backslash\{p,q,r,pq,qr\},\\&& \Sigma_6=R^+\backslash\{p,r,s,rs\},\ \Sigma_7=R^+\backslash\{p,q,r,s,pq,rs,qr,qrs\},\\&& \Sigma_8=R^+\backslash\{q,r,s,qr,rs,q2r,qrs,q2rs\},\ \Sigma_9=R^+\backslash\{p,q,r,s,pq,rs,qr,q2r,pqr,pq2r\},\\&& \Sigma_{10}=R^+\backslash\{p,q,r,s,pq,rs,qr,pqr,qrs,pqrs\},\ \Sigma_{11}=R^+\backslash\la p,q,r\ra,\ \Sigma_{12}=R^+\backslash\la q,r,s\ra,\\&& \Sigma_{13}=\{p2q2r,p2q2rs,p2q2r2s,p2q3rs,p2q3r2s,p2q4r2s,p3q4r2s,2p3q4r2s\},\\&&\Sigma_{14}=\{q2r2s,pq2r2s,p2q2r2s,p2q3r2s,p2q4r2s,p3q4r2s,2p3q4r2s\},\ \Sigma_{15}=\{2p3q4r2s\}.
\end{eqnarray*}
For $E\subset\Lg^*$, we define $G(E)=\{g.\xi\,|\,g\in G,\xi\in E\}$. Using  \cite{CP} and well-known results on nilpotent coadjoint orbits in $\Lg^*$ in characteristic $0$ (identified with nilpotent orbits in $\Lg$), we get that the nilpotent pieces in $\Lg^*$ defined in \cite{CP} are as follows
\begin{eqnarray}\label{eqnp}
H_i=G(S_i)\backslash\bigcup_{j\in I_i}G(S_j), \ i=1,\ldots,16,
\end{eqnarray}
where
\begin{eqnarray*}
&&I_1=\{2,\ldots,16\}, I_2=\{3,\ldots,16\},I_3=\{4,\ldots,16\},\ I_4=I_5=\{6,\ldots,16\},\ I_6=\{7,\ldots,16\},\\
&&\ I_7=\{8,\ldots,16\},\ I_8=\{10,12,\ldots,16\},\ I_9=\{10,\ldots,16\},\ I_{10}=\{12,13,14,15,16\},\\&&
 \ I_{11}=I_{12}=\{13,\ldots,16\},\ I_{13}=\{14,15,16\},\  I_{14}=\{15,16\}, I_{15}=\{16\},\ I_{16}=\emptyset.
\end{eqnarray*}
Moreover \begin{equation}\label{eqnsub}
G(S_j)\subsetneqq G(S_i)\text{ for all }j\in I_i, \ i=1,\dots,16, 
\end{equation}
and $$\{\dim H_i,i=1,\ldots,16\}=\{48,46,44,42,42,40,38,36,36,34,30,30,28,22,16,0\}.$$
It is easy to see that 
\begin{eqnarray}\label{eqnr1}
&&\xi_i\in G(S_i)\text{ for }i\in\{1,\ldots,6\},\quad \xi_7\in G(S_4),\quad \xi_{i}\in G(S_{i-1})\text{ for }i\in\{8,\ldots,14\},\\&& \xi_{15}\in G(S_{12}),\quad  \xi_{i}\in G(S_{i-2})\text{ for } i\in\{16,17,18\}.\nonumber
\end{eqnarray} One can also verify that 
\begin{eqnarray}\label{eqnnin}&&\xi_4\notin G(S_5),\ \xi_5\notin G(S_4),\ \xi_7\notin G(S_5),\ \xi_9\notin G(S_9),\\
&& \xi_{10}\notin G(S_8),\ \xi_{12}\notin G(S_{12}),\ \xi_{13}\notin G(S_{11}),\ \xi_{15}\notin G(S_{11}).\nonumber
\end{eqnarray} For example, $\xi_7\notin G(S_5)$ since there exists no $w\in W$ such that $w(p)\in \Sigma_5$, $w(qr)\in\Sigma_5$, and $w(q2r2s)\in \Sigma_5$ if either $w(qrs)\notin \Sigma_5$, or both $w(q2r)\notin \Sigma_5$ and $w(q2rs)\notin \Sigma_5$.

It follows from (\ref{eqnp})-(\ref{eqnnin}) and dimension consideration that the nilpotent
pieces in $\Lg^*$ are as follows
\begin{eqnarray*}
&&H_i=C_i,\  i=1,2,3,5,6,\\
&&H_i=C_{i+1},\ i=7,8,9,10,11,13,\\
&&H_i=C_{i+2},\ i=14,15,16,\\
&&H_4=C_4\cup C_7,\\
&&H_{12}=C_{13}\cup C_{15},
\end{eqnarray*}
where $C_i$ is the $G$-orbit of $\xi_i$, $i=1,\ldots,18$.

By \cite{CP} and \cite{H2}, we have $\overline{H_i}=G(S_i)$. It follows that $\overline{C_4}=G(S_4)$, $\overline{C_{13}}=G(S_{12})$, and moreover $C_7\nsubseteq\overline{C_5}$, $C_{15}\nsubseteq\overline{C_{14}}$ in view of (\ref{eqnsub}) and (\ref{eqnnin}). We show that
\begin{eqnarray}
&&C_9\subset\overline{C_7},\qquad C_{12}\nsubseteq\overline{C_{7}}\label{eqn-1}\\
&&C_{17}\subset\overline{C_{15}},\qquad C_{16}\nsubseteq\overline{C_{15}}\label{eqn-2}.
\end{eqnarray}
Let $C,C'$ be two nilpotent coadjoint orbits and $\xi\in C\cap\Ln^*$. We have $C'\subset\overline{C}$ if and only if $C'\cap\overline{B.\xi}\neq\emptyset$ since $G/B$ is complete. Using   $(a)$ in \ref{ssec-cen}, one can show that
\begin{eqnarray}
&&\overline{B.\xi_7}=\{\xi=\sum_{\beta\in R^+\backslash\{q,r,s,rs\}}c_\beta e_\beta'\,|\,c_{pqr}c_{qrs}=c_{qr}c_{pqrs},\ c_{pqr}c_{q2rs}=c_{qr}c_{pq2rs},\label{eqn-4}\\&&\hspace{2.1in} c_{pq2rs}c_{qrs}=c_{q2rs}c_{pqrs}\}\nonumber\\
&&\overline{B.\xi_{15}}=\text{span}\{e_{\beta}',\ \beta\in\{p2q2r,pq2r2s,p2q2r2s,p2q4r2s,p3q4r2s,2p3q4r2s\}\}\label{eqn-3}.
\end{eqnarray}
Now (\ref{eqn-2}) follows from (\ref{eqn-3}) and $C_{9}\subset\overline{C_7}$ follows from (\ref{eqn-4}). Suppose that $C_{12}\subset\overline{C_7}$. Then there exists $g\in G$ such that $g.\xi_{12}\in\overline{B.\xi_7}$. Suppose that  $g\in BwB$. By (\ref{eqn-4}), $w(pqrs),\ w(q2rs)\in R^+\backslash\{q,r,s,rs\}$. It follows that \begin{eqnarray}\label{eqn5}
&&\{w(pqrs),\ w(q2rs)\}\in\{\{pqr,qrs\},\{qr,pqrs\},\{pqr,q2rs\},\\&&\hspace{2in}\{qr,pq2rs\}, \{pq2rs,qrs\},\{q2rs,pqrs\}\}.\nonumber
\end{eqnarray}
Suppose $\{w(pqrs),\ w(q2rs)\}=\{pqr,qrs\}$. Write $g.\xi_{12}=\sum c_\beta e_\beta'$. Then $c_{pqr}\neq 0, c_{qrs}\neq 0$. Since $g.\xi_{12}\in\overline{B.\xi_7}$, by (\ref{eqn-4}), there exist $\beta,\beta'\in R^+$ greater than $pqrs$ or $q2rs$ such that $\{w(\beta),w(\beta')\}=\{qr,pqrs\}$. But this is impossible. Similarly one shows that $\{w(pqrs),\ w(q2rs)\}$ can not equal 
any set of pairs in the right hand side of (\ref{eqn5}). This gives us a contradiction. Thus (\ref{eqn-1}) is proved.

It follows from the above discussion that the closure relation on nilpotent coadjoint orbits in $\Lg^*$  is as in Figure \ref{fig2.tag}.

\begin{figure}
\centering
\includegraphics[height=130mm]{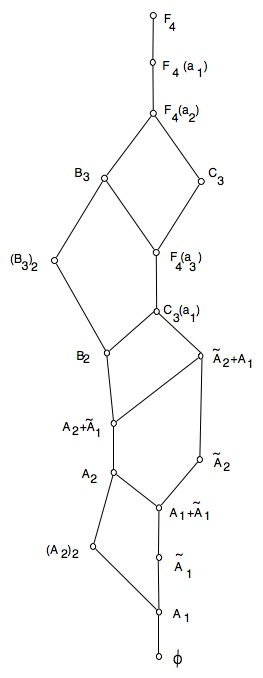}
\caption{Closure relation among nilpotent coadjoint orbits in $\Lg^*$, type $F_4$, $p=2$.}

\label{fig2.tag}
\end{figure}


\begin{thebibliography}{10}


\bibitem[C]{Ch}C. Chevalley, Sur certains groupes simples. (French)  T$\hat{o}$hoku Math. J. (2) 7 (1955), 14-66.
\bibitem[CP]{CP}M. C. Clarke and A. Premet, The Hesselink stratification of nullcones and base change. Invent. Math. 191 (2013), 631-669.

\bibitem[H1]{H2}W. H. Hesselink,  Nilpotency in classical groups over a field of characteristic 2. Math. Z. 166 (1979), no. 2, 165-181.

\bibitem[H2]{H1} W. H. Hesselink, Desingularizations of varieties of nullforms. Invent. Math. 55 (1979), no. 2, 141-163. 

\bibitem[HS]{HS}D. F. Holt and N. Spaltenstein, Nilpotent orbits of exceptional Lie algebras over algebraically closed fields of bad characteristic. J. Austral. Math. Soc. Ser. A 38 (1985), no. 3, 330-350.

\bibitem[K]{Kac}V. Kac,  Infinite-dimensional Lie algebras. Third edition. Cambridge University Press, Cambridge, 1990.


\bibitem[KW]{KW}V. Kac and B. Weisfeiler, Coadjoint action of a semi-simple algebraic group and the center of the enveloping algebra in characteristic $p$.  Nederl. Akad. Wetensch. Proc. Ser. A 79=Indag. Math.  38  (1976), no. 2, 136-151.






\bibitem[L1]{Lu3}G. Lusztig, A class of irreducible representations of a Weyl group.  Nederl. Akad.
Wetensch. Indag. Math. 41 (1979), no. 3, 323-335.

\bibitem[L2]{Lu1}G. Lusztig, Intersection cohomology complexes on a reductive group. Invent. Math. 75 (1984), no.2, 205-272.



\bibitem[L3]{Lu7}G. Lusztig, Notes on unipotent classes.  Asian J. Math.  1  (1997),  no. 1, 194-207.

\bibitem[L4]{L1}G. Lusztig,  Unipotent classes and special Weyl group representations. J. Algebra 321 (2009), no. 11, 3418-3449.


\bibitem[L5]{L4}G. Lusztig,  Unipotent elements in small characteristic, IV. Transform. Groups 15 (2010), no. 4, 921-936.

\bibitem[LS]{LS1}G. Lusztig and N. Spaltenstein, Induced unipotent classes. J. London Math. Soc. (2) 19 (1979), no. 1, 41-52.

\bibitem[PS]{PS}A. Premet and S. Skryabin, 
Representations of restricted Lie algebras and families of associative L-algebras. 
J. Reine Angew. Math. 507 (1999), 189-218.


\bibitem[Sh]{Sh}T. Shoji, The conjugacy classes of Chevalley groups of type $F_4$ over finite fields of characteristic $p\neq2$. J. Fac. Sci. Univ. Tokyo Sect. IA Math. 21 (1974), 1-17. 

\bibitem[Shi]{Shi} Ken-ichi Shinoda, The conjugacy classes of Chevalley groups of type $F_4$ over finite fields of characteristic 2. J. Fac. Sci. Univ. Tokyo Sect. I A Math. 21 (1974), 133-159.

\bibitem[Sp1]{Spa2}N. Spaltenstein, On the fixed point set of a unipotent element on the variety of Borel subgroups. Topology 16 (1977), no. 2, 203-204.
\bibitem[Sp2]{Spal}N. Spaltenstein, Nilpotent Classes and Sheets of Lie
Algebras in Bad Characteristic. Math. Z. 181 (1982), 31-48.
\bibitem[Sp3]{S5}N. Spaltenstein, Classes unipotentes et sous-groupes de Borel. (French) [Unipotent classes and Borel subgroups] Lecture Notes in Mathematics, 946. Springer-Verlag, Berlin-New York, 1982.

\bibitem[Sp4]{Sp3}N. Spaltenstein,  On the generalized Springer correspondence for exceptional groups. Algebraic groups and related topics (Kyoto/Nagoya, 1983), 317-338, Adv. Stud. Pure Math., 6, North-Holland, Amsterdam, 1985.

\bibitem[Sp5]{Sp4} N. Spaltenstein, Nilpotent classes in Lie algebras of type $F_4$ over fields of characteristic 2. J. Fac. Sci. Univ. Tokyo Sect. IA Math. 30 (1984), no. 3, 517-524.



\bibitem[St]{st}R. Steinberg, Conjugacy classes in algebraic groups. Lecture Notes in Mathematics, Vol. 366. Springer-Verlag, Berlin-New York, 1974.

\bibitem[S]{S} U. Stuhler,  Unipotente und nilpotente Klassen in einfachen Gruppen und Liealgebren vom Typ G2. (German) Nederl. Akad. Wetensch. Proc. Ser. A 74=Indag. Math. 33 (1971), 365-378.

\bibitem[X1]{X5}T. Xue, Nilpotent orbits
in the dual of classical Lie algebras in characteristic 2 and the
Springer correspondence. Represent. Theory 13 (2009), 609-635
(electronic).

\bibitem[X2]{X2}T. Xue, On unipotent and nilpotent pieces for classical groups. J. Algebra
349 (2012), no. 1, 165-184.
\bibitem[X3]{X3}T. Xue, Combinatorics of the Springer correspondence for classical Lie algebras and their duals in characteristic 2. Adv. Math. 230 (2012) 229-262.
\bibitem[X4]{X4}T. Xue, Nilpotent elements in the dual of odd orthogonal algebras.  Transform. Groups. 17 (2012), no. 2 (2012), 571-592.
\end{thebibliography}
\end{document}